\documentclass[11pt]{article}
\usepackage[latin1]{inputenc}
\usepackage[dvips]{graphicx}
\usepackage{latexsym}
\usepackage{amsmath,amssymb,amsfonts,amsthm}
\usepackage{pdfsync}
\usepackage{hyperref}
\usepackage{xcolor}

\newcommand{\nrn}{\rightarrow+\infty}
\newcommand{\xrn}{\xrightarrow}

\newcommand{\ER}{\mathbb {R}}\newcommand{\EN}{\mathbb {N}}
\newcommand{\PE}{\mathbb {P}}
\newcommand{\Z}{\mathbb {Z}}
\newcommand{\ES}{\mathbb{E}}
\newcommand{\abs[1]}{\lvert#1\rvert}

\newcommand{\psg}{\langle }
\newcommand{\psd}{\rangle }
\newcommand{\bx}{\bar{X}}

\newcommand{\Gn}{{n\gamma}}
\newcommand{\Gnp}{{(n+1)\gamma}}
\newcommand{\un}{\underline}
\newcommand{\gm}{\gamma}
\newcommand{\ab}{\alpha}

= 0 cm \evensidemargin = -0.3 cm \marginparwidth = 0 cm
\newtheorem{theorem}{ \textnormal{\bf{T\scriptsize{HEOREM}}}}
\newtheorem{prop}{\textnormal{\bf{P\scriptsize{ROPOSITION}}}}

\newtheorem{lemme}{\textnormal{\bf{L\scriptsize{EMMA}}}}

\theoremstyle{definition}
\newtheorem{definition}{\textnormal{\bf{D}\scriptsize{EFINITION}}}
\theoremstyle{remark}

\newtheorem{Remarque}{\textnormal{\bf{R\scriptsize{EMARK}}}}

\author{Serge Cohen\footnote{ Institut de Mathématiques de  Toulouse CNRS  UMR 5219, Université de Toulouse $\&$ Université Paul Sabatier 118, route de Narbonne 31062 Cedex 9  E-mail: \texttt{Serge.Cohen@math.univ-toulouse.fr}},
Fabien Panloup\footnote{ Institut de Mathématiques de  Toulouse CNRS UMR 5219, Université de Toulouse $\&$ INSA Toulouse, 135, Avenue de Rangueil, 31077 Toulouse Cedex 4, E-mail: \texttt{fabien.panloup@math.univ-toulouse.fr}},
Samy Tindel\footnote{Institut {\'E}lie Cartan Nancy, Universit\'e de Lorraine, B.P. 239,
54506 Vand{\oe}uvre-l{\`e}s-Nancy Cedex, France, E-mail: \texttt{samy.tindel@univ-lorraine.fr}}
}
\title{\textbf{ Approximation of stationary solutions to SDEs driven by multiplicative fractional noise}}
\begin{document}
\maketitle
\begin{abstract} 
In  a previous paper, we studied the ergodic properties of an Euler scheme of a stochastic differential equation
with a Gaussian additive noise in order to approximate the stationary regime of such an equation. We now consider the case of multiplicative noise when the Gaussian process is a fractional Brownian Motion with Hurst parameter $H>1/2$ and obtain some (functional) convergence properties of some empirical measures of the Euler scheme to the  stationary solutions of such SDEs. 
\end{abstract}


\noindent \textit{AMS classification (2000)}: 60G10, 60G15, 60H35.
\section{Introduction}
\noindent 
Stochastic Differential Equations (SDEs) driven by a fractional Brownian motion (fBm) have been introduced to model 
random evolution phenomena whose noise has long range dependence properties. 
Indeed, beyond the historical motivations in Hydrology and Telecommunication for the use of fBm (highlighted e.g in \cite{MVn68}), recent applications of dynamical systems driven by this process include challenging issues in Finance \cite{Gua06}, Biotechnology~\cite{Odde-al96} or Biophysics~\cite{Jeon-al11,Kou08}. As a consequence,  SDEs driven by fBm have been widely studied in a finite-time horizon during the last decades, and the reader is referred to~\cite{Nualart02,Coutin12} for nice overviews on this topic.

In a somehow different direction, the study of  the long-time behavior (under some stability properties) for fractional SDEs has been developed by Hairer (see \cite{hairer,hairer09}) and Hairer and Ohashi \cite{hairer2}, 
who built a way to define stationary solutions of these a priori non-Markov processes and to extend some of the tools of the Markovian theory to this setting. See also~\cite{Arnold98,crauel,GKN09} for another setting called 
random dynamical systems. The current article fits into this global aim, and starts from the following observation: 
the knowledge of the stationary regime being important for applications and essentially inaccessible in an explicit form, we propose to  build and to study a procedure for its approximation in the case of SDEs driven by fBm with a Hurst parameter $H>1/2$. This paper is following 
a similar previous work for  SDEs driven by more general noises 
but in the specific additive case  (see~\cite{cohen-panloup}).

More precisely, we deal   with an  $\ER^d$-valued process $(X_t)_{t\ge0}$ which is a
solution to the following SDE
\begin{equation}\label{fractionalSDE0}
dX_t=b(X_t)dt+ \sigma(X_t)dB_t^H
\end{equation}
where $b:\ER^{d}\rightarrow\ER^d$ and $\sigma:\ER^{d}\rightarrow \mathbb{M}_{d,q} $  are (at least) continuous functions, and where $ \mathbb{M}_{d,q}$ is the  set of $d\times q$ real matrices. In~\eqref{fractionalSDE0}, $(B^H_t)_{t\ge0}$ is  a $q$-dimensional $H$-fBm  and for the sake of simplicity we assume
 $ \frac1{2} < H < 1$, which allows in particular to invoke Young integration techniques in order to define stochastic integrals with respect to $B^H$. Compared to~\cite{cohen-panloup} we handle here a fairly general diffusion coefficient $\sigma$, instead of the constant one considered previously. Classically the noise is called multiplicative in this setting, whereas it is called additive when $\sigma $ is constant.\\
Under some Hölder regularity assumptions on the coefficients (see Section \ref{mainresult} for details), (strong) existence and uniqueness hold for the solution to \eqref{fractionalSDE0} starting from $x_0\in\ER^d$. Classically for  any stochastic differential equation, a natural question arises: if we assume that some Lyapunov assumptions hold on the drift term, does it imply that $(X_t)_{t\ge0}$ has some convergence properties to a steady state  when $t\rightarrow+\infty$ ?\\
This question implies in particular to define rigorously a concept of steady state. For  equation \eqref{fractionalSDE0}, this work has been done in \cite{hairer2}: using the fact that, owing to the Mandelbrot representation, the evolution of the fBm can be represented through a Feller transition on a functional space ${\cal S}$, the authors show that a solution to \eqref{fractionalSDE0} can be built as the first coordinate of an homogeneous Markov process on the product space $\ER^d\times {\cal S}$. As a consequence, stationary regimes associated with   \eqref{fractionalSDE0} can be naturally defined as the first projection of invariant measures of this Markov process.
Furthermore,  the authors of \cite{hairer2} develop some specific theory on strong Feller and irreducibility properties to prove uniqueness of invariant measures in this context.\\
In the current article, our aim is to propose a way to approximate numerically the stationary solutions to equation  \eqref{fractionalSDE0}. To this end, we study some empirical occupation measures related to an Euler type approximation of  \eqref{fractionalSDE0} with step $\gamma > 0$. We show that, under some Lyapunov assumptions, this sequence of empirical measures converges almost surely to the distribution of the stationary solution of the discretized equation (denoted by $\nu^\gamma$) and that, when $\gamma\rightarrow0^+$, $\nu^\gamma$ converges in turn to the distribution of the stationary solution of \eqref{fractionalSDE0}.   This approach is the same as  in \cite{cohen-panloup}. However, the introduction of multiplicative noise has some important consequences on the techniques for proving the long-time stability of the Euler scheme. In particular, the main difficulty is to show that the long-time control of the dynamical system can be achieved independently of $\gamma$. In \cite{cohen-panloup}, this problem has been solved with the help of explicit computations 
for an  Ornstein-Uhlenbeck type process. Because  the noise is multiplicative the computations of~\cite{cohen-panloup} are not feasible anymore and  we   use specific tools to obtain uniforms controls of discretized integrals with respect to the fBm.
Before going more precisely to the heart of the matter, let us mention that the numerical approximation of the stationary regime by occupation measures of Euler schemes is a classical  problem in a Markov setting including diffusions and Lévy driven SDEs (see $e.g.$ \cite{talay,LP1,LP2,lemaire2,PP1,panloup1}).

\section{Framework and main results}\label{mainresult}
This section is firstly devoted to specify the setting under which our computations will be performed. Namely, we give an account on differential equations driven by fractional Brownian motion and their related ergodic theory. Once this framework is recalled, we shall be able to state our main results.

\subsection{FBm and H\"older spaces}
For some fixed $H\in(\frac{1}{2},1)$, we consider $(\Omega,\mathcal{F},\PE)$ the canonical probability space associated with the fractional
Brownian motion indexed by $\ER$ with Hurst parameter $H$. That is,  $\Omega=\mathcal{C}_0(\ER)$ is the Banach space of continuous functions
vanishing at $0$ equipped with the supremum norm, $\mathcal{F}$ is the Borel sigma-algebra and $\PE$ is the unique probability
measure on $\Omega$ such that the canonical process $B^H=\{B^H_t=(B^{H,1}_t,\ldots,B^{H,q}_t), \; t\in \ER\}$ is a fractional Brownian motion with Hurst
parameter $H$.
In this context, let us recall that $B^H$ is a $q$-dimensional centered Gaussian process such that $B^H_0=0$, 
whose coordinates are independent and satisfy
\begin{equation}\label{eq:var-increm-fbm}
\mathbb{E}\left[  \left( B_t^{H,j} -B_s^{H,j}\right)^2\right]= |t-s|^{2H}, \quad\mbox{for}\quad s,t\in\ER.
\end{equation}
In particular it can be shown, by a standard application of Kolmogorov's criterion, that $B^H$ admits a continuous version
whose paths are $\theta$-H\"older continuous for any $\theta<H$.

Let us be more specific about the definition of H\"older spaces of continuous functions. Namely, our driving process $B^H$ lies into a space ${\cal C}^\theta$ defined as follows:
we denote by ${\cal C}^\theta(\ER_+,\ER^d)$ the set of functions $f:\ER_+\rightarrow\ER^d$ such that 
 $$\forall T>0,\quad \|f\|_{\theta,T}=\sup_{0\le s< t\le T}\frac{ |f(t)-f(s)|}{(t-s)^\theta}<+\infty,$$
where the Euclidean norm is denoted by $|\, .\,|$.
We recall that ${\cal C}^\theta(\ER_+,\ER^d)$ can be made into a non-separable complete metric space, whenever endowed with the distance $\delta_{\theta}$ 
defined by
$$\delta_{\theta}(f,g)=\sum_{N\in\mathbb{N}}2^{-N}\left(1\wedge\left(\sup_{0\le t\le N}\|f(t)-g(t)\|+\|f-g\|_{\theta,N}\right)\right),$$
where $ x \wedge y =\min(x,y) \;\forall x,\,y \in \mathbb R.$ 
However, since separable spaces are crucial for convergence in law issues, we will work in fact with a smaller space $\bar{\cal C}^\theta(\ER_+,\ER^d)$: we say that a function $f$ in ${\cal C}^\theta(\ER_+,\ER^d)$ belongs to $\bar{\cal C}^\theta(\ER_+,\ER^d)$ if
\begin{equation}\label{eq:def-bar-C-theta}
\forall\, T>0,\quad \omega_{\theta,T}(f,\delta):=\sup_{0\le s<t<T,0\le |t-s|\le\delta}\frac{|f(t)-f(s)|}{|t-s|^\theta}\xrn{\delta\rightarrow0}0.
\end{equation}
$\bar{\cal C}^\theta(\ER_+,\ER^d)$ is a closed separable subspace of ${\cal C}^\theta(\ER_+,\ER^d)$.
 
\subsection{Differential equations driven by fBm}

  We recall now some results on existence and uniqueness of solutions of the stochastic differential equation~\eqref{fractionalSDE0}  starting from a deterministic point.
  
When $B^{H}$ is a fractional Brownian motion with Hurst parameter  $H>1/2$, equations of the form~\eqref{fractionalSDE0} are classically solved by interpreting the stochastic integral $\int_{0}^{t}\sigma(X_{u})\, dB_{u}^{H}$ as a Young integral (see e.g \cite{FV-bk}). The usual set of assumptions on the coefficients $b$ and $\sigma$ are then of Lipschitz and boundedness types.

Specifically, we recall the following definition of a $(1+\alpha)$-Lipschitz function:
\begin{definition}
Let  $\sigma:\ER^{d}\rightarrow \mathbb{M}_{d,q} $  be a ${\cal C}^{1}$ function and $ 0 < \alpha < 1$. We say that $\sigma $ is  $(1+\alpha)$-Lipschitz if  the following norm is finite:
\begin{equation}
  \label{eq:1+gammaLips}
  \|\sigma\|_{1 + \alpha}= \sup_{ x \in \mathbb R^d } \|D \sigma(x) \| + \sup_{x, y \in \mathbb R^d}\frac{|D \sigma(x)-D \sigma(y)|}{|x-y|^\alpha}.
\end{equation}
\end{definition}
With this definition the basic existence and uniqueness result in a finite horizon $[0,T]$  for $T>0$  for pathwise equations driven by $\theta$-H\"older functions  with $\theta>1/2$ can be found in~\cite{Coutin12,Lyons94}. 
Nevertheless in this article we are searching for stationary solutions, which have to be defined on $\ER_+.$
Moreover we use  ergodic results that require some damping effect of the continuous drift coefficient $b$. In order to quantify this notion, let us now introduce  a long-time stability assumption $(\mathbf{C})$. Namely,  let ${\cal E}\!{\cal Q}(\ER^d)$ denote the set of {\em Essentially Quadratic} functions, that is ${\cal C}^2$-functions
$V:\ER^d\rightarrow (0,\infty)$ such that 
\[
\liminf_{|x|\rightarrow+\infty} \frac{V(x)}{|x|^2}>0, \qquad \abs[\nabla V]\le C \sqrt{V}\quad \mbox{  and }\quad D^2V \mbox{ is  bounded.}
 \]
\noindent Note that any element  $V\in {\cal E}\!{\cal Q}(\ER^d)$ is continuous, and thus attains its positive minimum $\underline{v}>0$ so that,  for any $A, r>0$, there exists a real constant $C_{_{A,r}}$ such that $A+V^r \le C_{_{A,r}}V^r$.

With these notions in mind, our standing assumptions on the coefficients $b$ and $\sigma$ are summarized as:

 \noindent $\mathbf{(C)}$ The map $\sigma$ is assumed to be a  bounded Lipschitz continuous function.
Moreover we suppose that there exists $V\in {\cal E}\!{\cal Q}(\ER^d)$ such that
 \begin{itemize}
 \item[(i)]$ \forall x\in\ER^d \quad |b(x)|^2\le V(x) \; , $
 \item[(ii)] and such that for  $\beta\in\ER$ and $\alpha>0$ the following relation holds: 
 $$\forall x\in\ER^d\quad \psg \nabla V(x),b(x)\psd\le  \beta-\alpha V(x).$$
 \end{itemize}

\begin{prop}\label{prop:exist-uniq-smooth-coeff}
Let us suppose that in addition to assumption $(\mathbf{C}),$ $ b $ is Lipschitz continuous and that $\sigma$ is $(1+\alpha)$-Lipschitz
with  $ \alpha > \frac1{H} - 1$. Then

\noindent\emph{(i)} 
For any deterministic function $ B \in {\cal C}^\theta(\ER_+,\ER^q) $ with $ \theta > \frac12, $ and any $ x_0 \in \ER^d ,$  there exists a unique solution  $X\in {\cal C}^\theta(\ER_+,\ER^d) $ of
\begin{equation}
  \label{eq:sd-det}
   X_t= x_0 + \int_{0}^{t} b(X_u) du + \int_{0}^{t} \sigma(X_u)dB_u,
\end{equation}
where the integrals are interpreted in the Riemann-Stieljes sense.

\noindent\emph{(ii)}
Let us set  $X\equiv \Phi(x_0,B)$, so that $ \Phi(x_0,B) $  satisfies
$$ 
\Phi(x_0,B)_t = x_0 + \int_0^t b(\Phi(x_0,B)_s) ds +   \int_0^t  \sigma(\Phi(x_0,B)_s) d B_s. 
$$ 
Then the so-called Itô map $  \Phi $ is continuous from $ \ER^d \times  {\cal C}^\theta(\ER_+,\ER^q) $ into $  {\cal C}^\theta(\ER_+,\ER^d)$.
\end{prop}
\begin{Remarque}
Proposition \ref{prop:exist-uniq-smooth-coeff} is not completely standard, when $b$ is not bounded, and we haven't been able to find a specific reference giving an equivalent statement in the literature. Namely the case of bounded smooth coefficients $b$ and $\sigma$ is handled e.g in~\cite{Coutin12,Lyons94}. If we move to the case of a dissipative coefficient $b$, an existence and uniqueness result is available in~\cite{hairer2}. Nevertheless, this result also assumes that the derivatives of $b$ are bounded. Assumption  $(\mathbf{C})\emph{(i)}$ implies that $b$ is sublinear.With the boundedness and Lipschitz assumption on $\sigma$ assumed in  $(\mathbf{C}),$  the proof of the existence of a global solution of this stochastic equation and of the continuity of the Itô map is a consequence of Young and Gronwall inequalities. 
\end{Remarque}

 \noindent
\subsection{Ergodic theory for SDEs driven by fBm}
We can now define the  solution of the stochastic differential equation starting from a random variable $X_0.$
Since the It\^o map of Proposition~\ref{prop:exist-uniq-smooth-coeff} is used in the following definition 
we have to suppose that  in addition to assumption $(\mathbf{C}),$ $ b $ is Lipschitz continuous and that $\sigma$ is $(1+\alpha)$-Lipschitz with  $ \alpha > \frac1{H} - 1$. 
\begin{definition}
\label{def:sol-eds-random-initial-condition} 
Let  $B^H$ be  a fractional Brownian motion with $H> \frac12.$ A process $(X_t)_{t \in \ER_+  }$
is called a solution of  equation \eqref{fractionalSDE0} driven by $B^{H}$ starting at $ X_0,$
if for every   $ 1/2 < \theta < H < 1,$ $(X_t)_{t \in \ER_+  }$ is almost surely  ${\cal C}^\theta(\ER_+,\ER^d)$-valued and if $X= \Phi(X_0,B^H)$, almost surely.
\end{definition}
We now have  all the tools to define rigorously a stationary solution to the SDEs driven by fBm. In the following definition and further on  we use the notation 
$\theta_t:\omega\mapsto\omega(t+.)$ for every $t\ge0$ for the time-shift .
\begin{definition}
\label{def:stat-sol}  Let  $(X_t)_{t\ge0}$ denote an $\ER^d$-valued solution to \eqref{fractionalSDE0} in the sense of Definition~\ref{def:sol-eds-random-initial-condition}. Let $\nu$ denote the distribution of $(X_t)_{t\ge0}$ on ${\cal C}^\theta(\ER_+,\ER^d).$ Then, $\nu$ is   called a stationary solution of  \eqref{fractionalSDE0} if it is invariant under the time-shift. Such a stationary solution is called adapted, if for $ 0 \le t $ the processes $(X_s)_{0\le s\le t}$ and $ (B^H_{s})_{s \ge t} $ are  conditionally independent given $ (B^H_{s})_{s \le t} $. 
\end{definition}
Please note that there is an abuse of language in the preceding definition. The distribution of a process $(X_t)_{t\ge0}$ on ${\cal C}^\theta(\ER_+,\ER^d)$  cannot determine alone  if  $(X_t)_{t\ge0}$ is a solution of  \eqref{fractionalSDE0} in the sense of Definition~\ref{def:sol-eds-random-initial-condition}. We need the distribution of the pair $(X_t,B^H)_{t\ge0}$ to know if   $X= \Phi(X_0,B^H)$, almost surely. In particular it is not possible to take $X_0$ independent of $(B^H)_{t\ge0}$ in general as remarked in Proposition 5 of~\cite{cohen-panloup}. Nevertheless we consider as in the Definition 2.4 in~\cite{hairer2} that two distributions $(X^1_t,B^H)_{t\ge0})$ and $ (X^2_t,B^H)_{t\ge0})$ on ${\cal C}^\theta(\ER_+,\ER^d) \times {\cal C}^\theta(\ER_+,\ER^q) $ solutions of \eqref{fractionalSDE0} are equivalent if the distribution of $ X^1$ and of $X^2$ are  the same.  These definitions are  the same as  definitions in~\cite{hairer2} that come
from Stochastic Dynamical Systems (SDS). In particular, we require adaptedness of solutions. Compared to Random Dynamical Systems (RDS) (see~\cite{Arnold98} for an introduction), this property is specific to SDS and is strongly linked 
to the fact that for such dynamical systems, one can associate a Markovian structure (with an enlargement of the space).
Here, the main consequence is that the uniqueness of the stationary solution can be obtained through the criterions of uniqueness of
the invariant distribution of this associated Markov process. Such results will be stated later.

\noindent Let $\gamma$ be a positive number, we will now discretize equation~\eqref{fractionalSDE0} as follows, for every $n \ge 0,$ 
 \begin{equation}
\label{fractionalSDE0-disc}
  Y_t^{\gamma}=Y_{\Gn}^{\gamma}+{(t-n\gamma)}b(Y_{n\gamma}^{\gamma})+\sigma(Y_{n\gamma}^{\gamma})(B^H_t-B^H_{n\gm})\quad\forall t\in[n\gamma,(n+1)\gamma).
 \end{equation}
We set $$\un{t}_{\gamma}= \max\{\gamma k,\gamma k\le t,k\in\EN\}.$$
In fact, we will usually write $\un{t}$ instead of $\un{t}_{\gamma}$ in the sequel.
The discretization of~\eqref{fractionalSDE0} can also be introduced with the following 
discretization $ \Phi^{\gamma}: \ER^d \times {\cal C}^\theta(\ER_+,\ER^q) \mapsto  {\cal C}^\theta(\ER_+,\ER^d) $ of the Itô map~: 
\begin{equation}
  \label{eq:phi-gamma}
   \Phi^{\gamma}(x_0,B)_t := x_0 + \int_0^t b(\Phi^{\gamma}(x_0,B)_{\un{s}_{\gamma}}) ds +   \int_0^t  \sigma(\Phi^{\gamma}(x_0,B)_{\un{s}_{\gamma}}) d B_s.
\end{equation}
Please note that the definition of $  \Phi^{\gamma}$ does not involve any Riemann integration 
but only finite sums and that
\begin{equation}\label{eq:phi-gamma2}
 Y^{\gamma}=\Phi^{\gamma}(Y_0^{\gamma},(B^H_t)_{t\ge0}) \;\; a.s.
 \end{equation}
\noindent We now define \textit{stationary adapted solutions of \eqref{fractionalSDE0-disc}} in the spirit of the Definition \ref{def:stat-sol}. 
\begin{definition}
\label{def:stat-sol-disc}
Let $B^H$ denote a fractional Brownian motion with $H>1/2$ and let $X^\gamma$ be defined by $X^{\gamma}=\Phi^{\gamma}(X_0^{\gamma},(B^H_t)_{t\ge0})$. The distribution $\nu^\gamma$ of $X^\gamma$ on  ${\cal C}^\theta(\ER_+,\ER^d)$
is then called an adapted solution of \eqref{fractionalSDE0-disc} if the processes $(X_s^\gamma)_{0\le s\le t}$ and $ (B^H_{s})_{s \ge t} $
are  conditionally independent given $ (B^H_{s})_{s \le t} $. 
We will say that $\nu^\gamma$ is stationary if it is invariant by the shift maps $(\theta_{k\gamma})_{k\in\EN}.$  
\end{definition}

 Note that in this definition, there is a  slight abuse of language since we do not require the invariance by the shift maps $\theta_{t}$ for every $t\ge0$, but only when $t=k\gamma$, $k\in\EN$.

Let us introduce  the following uniqueness assumption for $\nu^\gamma$ and $\nu$:

\noindent $(\mathbf{S^\gamma})$ $(\gamma\ge0)$: There is at most one adapted stationary solution to \eqref{fractionalSDE0} (resp. to 
\eqref{eq:phi-gamma2}) if $\gamma=0$ (resp. if $\gamma>0$).

\noindent For $(\mathbf{S^0})$, we refer to Theorem 1.1. of  \cite{hairer2}.  When $\gamma>0$, we have the following proposition:
\begin{prop} 
\label{unicitemesinvariante}
Let $H\in(1/2,1)$. Assume that $d=q$ and that $b$ and $\sigma$ are ${\cal C}^2$-functions.
Assume that $\sigma$ is invertible and that $\sup_{x\in\ER^d} \sigma^{-1}(x)<+\infty$. Then, $(\mathbf{S^\gamma})$ holds for 
every $\gamma>0$.
\end{prop}
\noindent The proof, which is an application of \cite{hairer09}, is done in the appendix. 

\noindent Let us now focus on the construction of the approximation. We denote by $(\bx_t^{\gamma})_{t\ge0}$ the continuous-time Euler scheme defined by  $\bx_0^\gamma=x\in\ER^d$ and for every $n\ge0$
\begin{equation}
  \label{eq:euler-scheme-x}
  \bx_t^{\gamma}=\bx_{\Gn}^{\gamma}+{(t-n\gamma)}b(\bx_{n\gamma}^{\gamma})+\sigma(\bx_{n\gamma}^{\gamma})(B^H_t-B^H_{n\gm}) \quad\forall t\in[\Gn,\Gnp). 
\end{equation}
The process $ (\bar{X}_t^{\gamma})_{t\ge0}$ is a solution to~\eqref{fractionalSDE0-disc} such that $ \bar{X}_0^{\gamma}=x.$
In order to alleviate the notations and, when it is not confusing, we will usually write  $\bar{X}_{t}$  instead of $\bar{X}_t^{\gamma}$.
Now, we define a sequence of random
 probability measures $({\cal P}^{(n,\gamma)}(\omega,d\alpha))_{n\ge 1}$  on $\bar{{\mathcal C}}^\theta(\ER_+,\ER^d)$ with $\theta<H$  (recall that $\bar{{\mathcal C}}^\theta(\ER_+,\ER^d)$ is defined at \eqref{eq:def-bar-C-theta}) by
 $${\cal P}^{(n,\gamma)}(\omega,d\alpha)=\frac{1}{n}\sum_{k=1}^n
 {\delta}_{{\bar{X}}^\gamma_{\gamma(k-1)+.}(\omega)}
(d\alpha) 
$$ where 
 $\delta$ denotes the Dirac measure and where, for every $s\ge0$, ${\bar{X}}^\gamma_{s+.}:=({\bar{X}}^\gamma_{s+t})_{t\ge0}$ denotes the 
$s$-shifted process.\\  
We are now able to state the main theorem of this article:
\begin{theorem}\label{principal1} Let $1/2< \theta < H < 1$ and  assume $\mathbf{(C)}.$ 
If  $\mathbf{(S^\gamma)}$ holds for every $\gamma>0,$\\
\noindent (i)   then there exists $\gamma_0>0$ such that, for every $\gamma\in(0,\gamma_0)$,
$$\lim_{n\rightarrow+\infty}{\cal P}^{(n,\gamma)}(\omega,d\alpha)=\nu^\gamma(d\alpha)\quad a.s. \; \textnormal{ when  $n\rightarrow+\infty$},$$
where the convergence is  for the weak topology induced by $\bar{{\mathcal C}}^\theta(\ER_+,\ER^d)$ and where $\nu^\gamma$ is the stationary solution of \eqref{fractionalSDE0-disc}.\\
\noindent (ii) If additionally, $b$ is Lipschitz continuous, $\sigma$ is $(1+\alpha)$-Lipschitz with $\alpha>\frac{1}{H}-1$ and if $\mathbf{(S^0)}$ holds, then
$$\lim_{\gamma\rightarrow0} \nu^\gamma(d\alpha) =\nu(d\alpha)\quad a.s.$$
where the convergence is  for the weak topology induced by $\bar{{\mathcal C}}^\theta(\ER_+,\ER^d)$ and where $\nu$ denotes the  adapted stationary solution of  \eqref{fractionalSDE0}.
\end{theorem}
\begin{Remarque}  Note that some extensions can be deduced from the proof of this theorem. 
First, remark that this result implies in particular that 
$$\lim_{\gamma\rightarrow0^+}\,\lim_{n\rightarrow+\infty}{\cal P}_0^{(n,\gamma)}(\omega,dy)=\nu_0(dy)\quad a.s.$$
where
$${\cal P}_0^{(n,\gamma)}(\omega,dy)=\frac{1}{n}\sum_{k=1}^n\delta_{\bar{X}^\gamma_{(k-1)\gamma}(dy)}$$
and $\nu_0(dy)$ denotes the initial distribution of the stationary solution $\nu$ of \eqref{fractionalSDE0}. This marginal procedure will be numerically tested in Section 6.\\
Also note that some extensions can be deduced from the proof of this theorem. First, when uniqueness fails for the stationary solutions, the preceding result is replaced by

\begin{theorem}\label{principal} Assume $\mathbf{(C)}$.
\smallskip

\noindent  1.  Then, there exists $\gamma_0>0$ such that for every $\gamma\in(0,\gamma_0)$, $({\cal P}^{(n,\gamma)}(\omega,d\alpha))_{n\ge1}$ is $a.s.$ tight on $\bar{{\mathcal C}}^\theta(\ER_+,\ER^d),$ for every $1/2< \theta < H < 1$. Furthermore,
 every weak limit is a stationary adapted solution of~\eqref{fractionalSDE0-disc}.\\
 2.  If additionally, $b$ is Lipschitz continuous, $\sigma$ is $(1+\alpha)$-Lipschitz with $\alpha>\frac{1}{H}-1,$ set
 $${\cal U}^{\infty,\gamma}(\omega):=\{\textnormal{weak limits of 
$({\cal P}^{(n,\gamma)}(\omega,d\alpha))$}\}.$$
 Then there exists $\gamma_1\in(0,\gamma_0)$ such that $({\cal U}^{\infty,\gamma}(\omega))_{\gamma\le\gamma_1}$ is a.s. tight in   $\bar{{\mathcal C}}^\theta(\ER_+,\ER^d),$ and any weak limit when $\gamma\rightarrow0$ of $({\cal U}^{\infty,\gamma}(\omega))_{\gamma\le\gamma_1}$ is an adapted  stationary solution of~\eqref{fractionalSDE0}. 
 \end{theorem} 
 \end{Remarque}
\begin{Remarque}  From the very definition of weak convergence, the preceding assertions imply that the convergence  of $({\cal P}^{(n,\gamma)}(\omega,d\alpha))_{n,\gamma}$ holds for bounded continuous functionals $F:\bar{{\mathcal C}}^\theta(\ER_+,\ER^d)\rightarrow\ER$.
In fact, this convergence can be extended for arbitrary $T > 0$ to some non-bounded continuous functionals $F:  \bar{{\mathcal C}}^\theta([0,T],\ER^d)\rightarrow\ER$.
Actually, setting $G(\alpha)=\sup_{t\in[0,T]}V(\alpha_t)$, we easily deduce from inequality \eqref{eq:sup-V} of Proposition \ref{lemme4} and Proposition \ref{prop:ntendinfty} that 
$$\sup_{\gamma\le\gamma_0}\limsup_{n\rightarrow+\infty}{\cal P}^{(n,\gamma)}(\omega,G^p)<+\infty\quad a.s.$$ for every $p>0$. By a uniform integrability argument, it follows 
\begin{prop}
The convergence properties of $({\cal P}^{(n,\gamma)}(\omega,d\alpha))$ extend to continuous functionals $F:  \bar{{\mathcal C}}^\theta([0,T],\ER^d)\rightarrow\ER$ such that there exists a constant $C$ such that for every $\alpha\in\bar{{\mathcal C}}^\theta([0,T],\ER^d)$, 
$$ |F(\alpha_t,0\le t\le T)|\le C\sup_{t\in[0,T]}V^p(\alpha_t)$$
 with $T>0$ and $p>0$.
\end{prop}
\end{Remarque}

\begin{Remarque} A third natural extension of Theorem \ref{principal1} consists in handling the case of an irregular fractional Brownian motion $B$ with Hurst index $1/4<H<1/2$. This extension is presumably within the reach of our technology on differential systems driven by fBm, but requires a huge amount of technical elaboration. Indeed, to start with, equation \eqref{fractionalSDE0} has to be defined thanks to rough paths techniques whenever $H<1/2$, and we refer to \cite{FV-bk} for a complete account on rough differential equations driven by Gaussian processes in general and fractional Brownian motion in particular. More importantly, as it will be observed in the next sections, our main result heavily relies on some thorough estimates performed on the discretized version \eqref{fractionalSDE0-disc} of equation \eqref{fractionalSDE0}. When $H>1/2$ this discretization procedure is based on an Euler type scheme, but the case $H<1/2$ involves the introduction of some L\'evy area correction terms of Milstein type (see \cite{Da07}) or products of increments of $B^H$ if one desires to deal with an implementable numerical scheme (cf. \cite{DNT}). This new setting has tremendous effects on the proof of Propositions \ref{lemme4} and \ref{prop:ntendinfty}. For sake of conciseness, we have thus decided to stick to the case $H>1/2$, and defer the rough case to a subsequent publication.

\end{Remarque}

The sequel of the paper is built as follows. The three next sections are devoted to the proof of Theorem \ref{principal1}.
In Section  \ref{section3}, we prove some preliminary results for the long-time stability of $({\cal P}^{(n,\gamma)}(\omega,d\alpha))_{n},$ when $ \gamma > 0.$ It is important to note that the controls established in this section are independent of $\gamma$ in order to obtain in the sequel a long-time control that does not explode when $\gamma\rightarrow0$. Then, in Section \ref{section4}, we obtain some tightness properties for $({\cal P}^{(n,\gamma)}(\omega,d\alpha))$ (in $n$ and $\gamma$) and, in Section \ref{section5}, we prove that the weak limits of this sequence are adapted stationary solutions.
Eventually, in Section \ref{section6}, we test numerically our algorithm for the approximation of the invariant distribution of a 
particular fractional SDE.\\
Note that in the proofs below, non-explicit constants are usually denoted by $C$ or $C_T$ (if a dependence to 
$T$ needs to be emphasized) and may change from line to line.
\section{Evolution control of $(\bar{X}_t^\gamma)$ in a finite horizon}\label{section3}
The main aim of this part is  to obtain a  finite-time control of $V(\bar{X}_T^\gamma)$ in terms of $V(\bar{X}_0^\gamma)$ which is independent of $\gamma$. 
This is the purpose of the first part of Proposition \ref{lemme4} below. 
In order to obtain some functional convergence results, we state in the second part a result  about the finite-time control of the Hölder semi-norm of $\bar{X}^\gamma$.

\begin{prop}\label{lemme4} Let $T>0$. Assume $\mathbf{(C)}$. Then, \\
(i) For every $ p \ge 1,$ there exist $\gamma_0>0$, $\rho\in(0,1)$ and a polynomial function $P_{p,\theta}:\ER\rightarrow\ER$ such that for every $\gamma\in(0,\gm_0]$,
\begin{align}\label{VX}
V^p(\bx_T^\gamma)\le \rho V^p(x)+P_{p,\theta}(\|B^H\|_{\theta,T}).
\end{align}
Furthermore, 
\begin{align}
\label{eq:sup-V}
\sup_{t\in[0,T]} V^p(\bx_t^\gm)\le C\left( V^p(x)+P_{p,\theta}(\|B^H\|_{\theta,T})\right).
\end{align}
(ii) For every $\theta\in(\frac{1}{2},H)$, $T>0$, and $\gamma\in(0,\gamma_0]$
\begin{equation}\label{eq:holder-bnd-with-V}
\sup_{0\le s<t\le T}\frac{|\bx^\gm_t -\bx^\gm_s|}{(t-s)^\theta}\le C_T\left(V(x)+\tilde{P}_\theta(\|B^H\|_{\theta,T})\right),
\end{equation}
where $\tilde{P}$ is another real valued polynomial function.
\end{prop}
The proof of this result is achieved in Subsection \ref{ss:22}. Before, we focus in Subsection \ref{ss:21} on the control 
of increments of some  discretized equations with non-bounded coefficients driven by $B^H.$ 
\subsection{Technical Lemmas}\label{ss:21}
Let us recall that, for every $t\ge0$, $\un{t}_{\gamma}=\gamma \max\{k\in\EN,\gamma k\le t\}.$
In the sequel, we will usually write $\un{t}$ instead of $\un{t}_\gamma$.

In the following lemmas, we will use the following notation:  for any element $(x(t))_{t\ge0}$ of ${\cal C}(\ER_+,\ER^d)$ and $T>0,$ $\theta>0$, $\gamma>0$, we define 
$$\|x\|_{\theta,\gamma}^{s,t}=\sup_{s\le u\le v\le t}\frac{|x(\un{v}_\gm)-x( \un{u}_\gm)|}{(\un{v}_\gamma-\un{u}_\gm)^\theta},$$
where we set by convention $\frac00 =0.$ 

\begin{lemme}\label{gronwall}
Assume that $b$ is a sublinear function, $i.e.$ that there exists $C>0$ such that for every $x\in\ER^d$, $|b(x)|\le C(1+|x|)$. Then, for every $T>0$, there exists a constant $C>0$ such that for every $s,t\in[0,T]$ with $s\le t$, for every $\gamma>0$, for every $\theta\in(0,H)$
$$|\bx_{{\un{t}}}^{\gamma}|\le \left(|\bx^\gamma_{\un{s}}|+C({\un{t}}-{\un{s}})+\|\bar{Z}^\gamma\|_{\theta,\gamma}^{{s},{t}}(\un{t}-\un{s})^\theta \right)\exp(C(\un{t}-\un{s})))$$ 
where
$$\bar{Z}^\gamma_t=\int_0^t\sigma(\bx^{\gm}_{{\un{s}}})dB^H_s.$$
\end{lemme}
\begin{proof}
First, from the very definition of  $(\bx^{\gm}_t)_{t\ge0}$, we have for every $s,t\in[0,T]$ with $s\le t$:
\begin{equation}\label{eq:ceuler}
\bx^{\gm}_{\un{t}}=\bx^{\gm}_{\un{s}}+\int_{\un{s}}^{\un{t}} b(\bx^{\gm}_{\un{u}})du+\bar{Z}^{\gm}_{\un{t}}-\bar{Z}^{\gm}_{\un{s}}.
\end{equation}
The function $b$ being sublinear, we deduce that
$$|\bx^{\gm}_{\un{t}}|=|\bx^{\gm}_{\un{s}}|+\|\bar{Z}^\gamma\|_{\theta,\gamma}^{{s},{t}}(\un{t}-\un{s})^\theta+C \int_{\un{s}}^{\un{t}}(1+|\bx^{\gm}_{\un{u}}|)du.$$
Setting $g_s(v)=|\bx_{\un{s}+v}|$, it follows that for every $v\in[0,\un{t}-\un{s}]$,
$$g_{{s}}(v)\le a+C \int_0^v g_{{s}}(u)du$$
with 
$a=|\bx^{\gm}_{\un{s}}|+\|\bar{Z}^\gamma\|_{\theta,\gamma}^{{s},{t}}(\un{t}-\un{s})^\theta +C(\un{t}-\un{s}).$
The result follows from Gronwall's lemma.
\end{proof}

The control of $B^H$-integrals is usually based on the so-called sewing Lemma (see $e.g.$ \cite{Coutin12,Feyel06}) which leads to a comparison of 
$\int_s^t f(x_u) dB_u^H$ with $f(x_t)(B^H_t-B^H_s)$. The following lemma can be viewed as a discretized version of such results: 
\begin{lemme}\label{lemme2}  Assume that $b$ is a sublinear function. Let $\gamma_0>0$ and  $(f_\gamma)_{\gamma\in(0,\gamma_0]}$ be a family of functions from 
$\ER_+\times\ER^{d}$ to $\mathbb{M}_{d,q}$ such that there exists $r\ge 0$ such that for every $T>0$, there exists $C_T>0$ such that $\forall\gamma\in(0,\gm_0],$
\begin{equation}\label{assump:f}
\;\forall (s,x),(t,y)\in[0,T]\times\ER^{d},\quad \|f_\gamma(t,y)-f_\gamma(s,x)\|\le C_T(1+|x|^r+|y|^r)(|t-s|+|y-x|).
\end{equation}
Let $(\bar{H}_t^{\gm})_{t\ge0}$ be defined by
$$\forall t\ge0,\quad \bar{H}_t^{\gm}=\int_0^t f_\gamma(\un{s}, \bx_{\un{s}}^\gm)d B^H_s.$$
Then, for every $\theta\in(\frac{1}{2},H)$, for every $T>0$, there exists $\tilde{C}_T>0$ such that for every $\gm\in(0,\gm_0]$, for every  $0\le s\le t\le T$,
\begin{equation}
  \label{eq:young}
|\bar{H}_{\un{t}}^{\gm}-\bar{H}_{\un{s}}^{\gm}-f_\gamma({\un{s}},\bx^\gamma_{{\un{s}}})(B^H_{\un{t}}-B^H_{\un{s}})|\le\tilde{C}_T(\un{t}-\un{s})^{2\theta}((1+|\bx_{\un{s}}^\gm|^{r+1}+(\|\bar{Z}^\gm\|_{\theta,\gamma}^{\un{s},\un{t}-\gamma})^{r+1})\|B^H\|_{\theta,T}.  
\end{equation}
\end{lemme}
\begin{proof}
Denoting by $\tilde{f}_{\gamma,\omega}$ the (random) function on $\ER_+$ $a.s.$ defined by  $\tilde{f}_{\gamma,\omega}(s)=f_\gamma(\un{s},\bar{X}_{\un{s}}^\gamma(\omega))$,  we can write: 
$$\bar{H}_{\un{t}}^{\gm}-\bar{H}_{\un{s}}^{\gm}-f_\gamma({\un{s}},\bx^\gamma_{{\un{s}}})(B^H_{\un{t}}-B^H_{\un{s}})=
\int_{\un{s}}^{\un{t}}  \tilde{f}_{\gamma,\omega}(u)-\tilde{f}_{\gamma,\omega}(s) d B_u^H.$$
Let $\theta\in(1/2,H)$ (so that $2\theta>1$). We use a classical Young estimate (see $e.g.$ \cite{Young36}, Inequality (10.9)), to get a upper bound for the left hand side of~\eqref{eq:young}. Let us recall the definition 
of $p$-variations.   For every $u, v\ge 0$ such that $u\le v$, for every $p>0$ and for every function $f:\ER_+\rightarrow\ER^d$, $$V_p(f,u,v)=\sup(\sum_{i=1}^n |f(t_i)-f(t_{i-1})|^p)^\frac{1}{p},$$
the supremum being taken over all subdivisions $(t_i)$ of $[u,v]$: $u=t_0< t_1< \ldots< t_n=v$.
Then using Young inequality we get 
\begin{equation}\label{youngestimate}
|\bar{H}_{\un{t}}^{\gm}-\bar{H}_{\un{s}}^{\gm}-f_\gamma({\un{s}},\bx^\gamma_{{\un{s}}})(B^H_{\un{t}}-B^H_{\un{s}})|
\le C V_{\frac{1}{\theta}} (\tilde{f}_{\gamma,\omega},\un{s},\un{t}-\gamma)V_{\frac{1}{\theta}} (B^H,\un{s},\un{t})
\end{equation}
where $C$ depends only on $\theta.$ 
Note that we could write $V_{\frac{1}{\theta}} (\tilde{f}_{\gamma,\omega},\un{s},\un{t}-\gamma)$ instead of $V_{\frac{1}{\theta}} (\tilde{f}_{\gamma,\omega},\un{s},\un{t})$ since $\tilde{f}_{\gamma,\omega}$ is constant on $[\un{t}-\gamma,\un{t})$.
We now control separately the two terms on the right-hand member.\\
Let $T>0$. Since for every $u,v\in[0,T]$,
$$|B^H_{v}-B^H_{u}|\le \|B^H\|_{\theta,T}|v-u|^\theta,$$
we first obtain that
\begin{equation}\label{pvariationB}
V_{\frac{1}{\theta}} (B^H,\un{s},\un{t})\le \|B^H\|_{\theta,T} (\un{t}-\un{s})^\theta.
\end{equation}
Second,  let $s,t\in[0,T]$ such that
$s\le t$ and consider a subdivision $(t_i)_{i=1}^n$ of $[\un{s},\un{t}-\gamma]$. By \eqref{assump:f}, we have
\begin{align*}
|\tilde{f}_{\gamma,\omega}(t_{i+1})-\tilde{f}_{\gamma,\omega}(t_i)|\le
C_T(1+|\bar{X}_{\un{t_i}}^\gamma|^r+|\bar{X}_{\un{t_{i+1}}}^\gamma|^r)(|\un{t_{i+1}}-\un{t_i}|+|\bar{X}_{\un{t_{i+1}}}^\gamma-\bar{X}_{\un{t_i}}^\gamma|).
\end{align*}
On the one hand, it follows from Lemma \ref{gronwall} that
$$1+|\bar{X}_{\un{t_i}}^\gamma|^r+|\bar{X}_{\un{t_{i+1}}}^\gamma|^r\le C_T\left(1+{|\bx_{\un{s}}^\gamma|}^r+(\|\bar{Z}^{\gamma}\|^{\un{s},\un{t}-\gamma}_{\theta,\gamma})^r\right).$$
On the other hand, since $b$ is a sublinear function, we have
$$|\bar{X}_{\un{t_{i+1}}}^\gamma-\bar{X}_{\un{t_i}}^\gamma|\le C_T(1+\int_{{t_i}}^{{t_{i+1}}}|\bx_{\un{u}}^\gamma|du+ |\bar{Z}^{\gamma}_{\un{t_{i+1}}}-\bar{Z}^{\gamma}_{\un{t_{i}}}|).$$
Then, using again Lemma \ref{gronwall} and the definition of $\|.\|^{\un{s},\un{t}-\gamma}_{\theta,\gamma}$, it follows that
$$ |\bar{X}_{\un{t_{i+1}}}^\gamma-\bar{X}_{\un{t_i}}^\gamma| \le C_T\left(1+|\bx^\gamma_{\un{s}}|+\|\bar{Z}^\gamma\|_{\theta,\gamma}^{{\un{s}},\un{t}-\gamma}\right)(\un{t_{i+1}}-\un{t_i})
+\|\bar{Z}^\gamma\|_{\theta,\gamma}^{{\un{s}},\un{t}-\gamma}(\un{t_{i+1}}-\un{t_i})^\theta.
$$
By a combination of the previous inequalities (and by the use of the Young inequality), we obtain 
\begin{align*}
|\tilde{f}_{\gamma,\omega}(t_{i+1})-\tilde{f}_{\gamma,\omega}(t_i)|\le
C_T(1+|\bar{X}_{\un{s}}^\gamma|^{r+1}+(\|\bar{Z}^{\gamma}\|^{\un{s},\un{t}-\gamma}_{\theta,\gamma})^{r+1})|\un{t_{i+1}}-\un{t_i}|^\theta.
\end{align*}
Since $ \sum_i (\un{t_{i+1}}-\un{t_i}) \le  \un{t}- \un{s},$ we deduce that 
$$V_{\frac{1}{\theta}} (\tilde{f}_{\gamma,\omega},\un{s},\un{t}-\gamma)
\le C_T(1+|\bar{X}_{\un{s}}^\gamma|^{r+1}+(\|\bar{Z}^{\gamma}\|^{\un{s},\un{t}-\gamma}_{\theta,\gamma})^{r+1})(\un{t}-\un{s})^\theta.$$
Finally, we plug  this control and \eqref{pvariationB} into \eqref{youngestimate} and the result follows.
\end{proof}
In the following lemma, we make use of Lemma \ref{lemme2} when $f_\gamma(t,x)=\sigma(x)$. In this particular case, we show below that 
we can deduce a control of the increments of $\bar{Z}^\gamma$ on an interval with random but explicit length $\eta(\omega)$ (which does not depend on $\gamma$). 
\begin{lemme}\label{lemme3} Let $\gamma_0$ be a positive number. Assume that $b$ is a sublinear function and that $\sigma$ is a bounded Lipschitz continuous function. Then,
for every $\theta\in(\frac{1}{2},H)$, for every $T>0$, there exists $C_T>0$, there exists a positive random variable
\begin{equation}\label{eq:valeureta}
\eta(\omega):=\left(\frac{1}{2}[(C_T \|B^H(\omega)\|_{\theta,T})^{-1}\wedge 1]\right)^\frac{1}{\theta}
\end{equation}
 such that $a.s$ for every $0\le s\le t\le T$ with $\un{t}-\un{s}\le\eta$, for every $\gamma\in(0,\gamma_0)$ 
$$ {|\bar{Z}^\gamma_{\un{t}}-\bar{Z}^\gamma_{\un{s}}|}\le  {(\un{t}-\un{s})^\theta}\left({2}\|\sigma\|_\infty+ C_T (1+|\bx^\gamma_{\un{s}}|)\eta^\theta\right)\|B^H\|_{\theta,T}$$
where $\|\sigma\|_\infty=\sup_{x\in\ER^d} \|\sigma(x)\|$. 

\end{lemme}
\begin{proof}
For every $l\ge 0$, set $t_l=\un{s}+\gamma l$ and $N_l=\|\bar{Z}^\gamma\|_{\theta,\gamma}^{\un{s},t_l}$. Owing to the definition of $\|.\|_{\theta,\gamma}^{\un{s},t_l}$, 
we have 
$$N_{l+1}\le N_l\vee \sup_{i\le l}\frac{\left|\bar{Z}^\gamma_{t_{l+1}}-\bar{Z}^\gamma_{t_i}\right|}{(t_{l+1}-t_i)^{\theta}}.$$
By Lemma \ref{lemme2} applied with $s=t_i$, $t=t_{l+1}$ and $f_\gamma(s,x)=\sigma(x)$ (and $r=0$),
$$\frac{\left|\bar{Z}^\gamma_{t_{l+1}}-\bar{Z}^\gamma_{t_i}\right|}{(t_{l+1}-t_i)^{\theta}}
\le \left(\|\sigma\|_\infty+ C_T(t_{l+1}-t_i)^{\theta}\left(1+|\bx^\gamma_{t_i}|+\|\bar{Z}^\gamma\|_{\theta,\gamma}^{\un{s},t_l}\right)\right)
\|B^H\|_{\theta,T}.$$
By Lemma \ref{gronwall} and the fact that $t\rightarrow\|\bar{Z}^\gamma\|_{\theta,\gamma}^{\un{s},t}$ is nondecreasing, it follows that
$$\sup_{i\le l}\frac{\left|\bar{Z}^\gamma_{t_{l+1}}-\bar{Z}^\gamma_{t_i}\right|}{(t_{l+1}-t_i)^{\theta}}
\le \left(\|\sigma\|_\infty+ C_T\left((1+|\bx^\gamma_{\un{s}}|)(t_{l+1}-\un{s})+\|\bar{Z}^\gamma\|_{\theta,\gamma}^{\un{s},t_l}(t_{l+1}-\un{s})^{\theta}\right)\right)
\|B^H\|_{\theta,T}.$$
Let $\rho$ be a positive number. If $t_{l+1}-\un{s}\le \rho$, we obtain that
$$N_{l+1}\le N_l\vee (\alpha_\rho+ \beta_\rho  N_l)$$
with 
\begin{align*}
&\alpha_\rho=\left(\|\sigma\|_\infty+ C_T((1+|\bx^\gamma_{\un{s}}|)\rho^\theta\right)\|B^H\|_{\theta,T}\quad\textnormal{and}\quad
&\beta_\rho=C_T \rho^\theta\|B^H\|_{\theta,T}.
\end{align*}
Let us now set $\rho=\eta(\omega)$ where $\eta(\omega)$ is defined by \eqref{eq:valeureta}. For this choice of $\rho$, we have $\beta_\eta\le\frac{1}{2}$. Then, the interval $[0,\alpha_\eta/(1-\beta_\eta)]$ being stable by the function $x\mapsto \alpha_\eta+\beta_\eta x$, we deduce that for every $l\in\EN$ such that $t_{l+1}-\un{s}\le \eta(\omega)$,
{$$N_l\le \frac{\alpha_\eta}{1-\beta_\eta}\le {2\alpha_\eta}.$$
Note that we used that $N_0$ belongs to $[0,\alpha_\eta/(1-\beta_\eta)]$ (since $N_0=0$).} The result follows.
\end{proof}

\subsection{Proof of Proposition \ref{lemme4}} \label{ss:22}

Proposition \ref{lemme4} is the main technical issue of our approximation result, and its proof is detailed here for sake of completeness. We shall first focus on establishing relation~\eqref{VX}  for $ p = 1$.  {The main 
difficulty is to prove that the noise component can be controlled in such a way that under the mean-reverting assumption,
we obtain a coefficient $\rho$ which is strictly lower than 1. (See in particular~\eqref{eq:entretildeeta}.) Note that this property  on $\rho$ will be crucial for the control of the sequence $(V(\bar{X}_{kT}))_{k\ge0}$.}

\smallskip
\noindent Then, we generalize this result to any $p>1$. Finally we handle the H\"older type bound of Proposition \ref{lemme4} item (ii). 
We now divide our proof in several steps.

\smallskip

\noindent
{\textit{Step 1: First upper-bound for $V(\bar{X}_{\underline{t}}^\gm)$ under the mean-reverting assumption.}}
Set $\Delta_n=B^H_{\gamma n} -B^H_{\gamma (n-1)}$. Owing to the Taylor formula,
\begin{align*}
V(\bx_{(n+1)\gm})&=V(\bx_{n\gm})+\gamma \psg\nabla V(\bx_{n\gm}),b(\bx_{n\gm})\psd+\psg \nabla V(\bx_{n\gm}),\sigma(\bx_{n\gm})\Delta_{n+1}\psd\\
&+\frac{1}{2} \sum_{i,j}\partial^2_{i,j} V(\xi_{n+1})(\bx_{(n+1)\gm}-\bx_{n\gm})_i(\bx_{(n+1)\gm}-\bx_{n\gm})_j.
\end{align*}
where $\xi_{n+1}\in[\bx_{n\gm},\bx_{(n+1)\gm}]$. Using assumption $\mathbf{(C)}${, equation \eqref{eq:euler-scheme-x} for $\bx_{(n+1)\gm}-\bx_{n\gm}$} and the boundedness of $D^2V$ and $\sigma$, we obtain
\begin{align}\label{entretildeeta2}
V(\bx_{(n+1)\gm})\le V(\bx_{n\gm})+\gamma(\beta-\alpha V(\bx_{n\gm}))+A_1(n+1)
+C(\gamma^2 V(\bx_{n\gm})+|\Delta_{n+1}|^2),
\end{align}
where 
$$A_1(n+1)=\psg \nabla V(\bx_{n\gm}),\sigma(\bx_{n\gm})\Delta_{n+1}\psd.$$
Set $\displaystyle\gamma_0=\frac{\alpha}{2C}$. For every $\gamma\in(0,\gamma_0]$, for every $n\ge0$, we have
\begin{align*}
V(\bx_{(n+1)\gm})\le V(\bx_{n\gm})(1-\frac{\alpha}{2}\gm)+A_1(n+1)
+(\beta\gamma+C|\Delta_{n+1}|^2).
\end{align*}
Then, iterating the previous inequality yields for every $s,t$ such that $s\le t$,
$$V(\bx_{{\un{t}}})\le V(\bx_{{\un{s}}})(1-\frac{\alpha}{2}\gm)^{\frac{{\un{t}}-{\un{s}}}{\gamma}}+\sum_{k=\frac{\un{s}}{\gamma}+1}^{{\frac{\un{t}}{\gm}}}
(1-\frac{\alpha}{2}\gm)^{\frac{{\un{t}}-{\un{s}}}{\gamma}-k}\left(A_1(k)+\beta\gamma+C|\Delta_k|^2\right).
$$
Using that $\log(1+x)\le x$ for every $x>-1$, we deduce that
\begin{equation}\label{eq:l41}
 V(\bx_{{\un{t}}})\le e^{-\frac{\alpha({\un{t}}-{\un{s}})}{2}}(V(\bx_{{\un{s}}})+|\bar{H}_{{\un{t}}}^\gm-\bar{H}_{{\un{s}}}^{\gamma}|)+
 \sum_{k=\frac{\un{s}}{\gamma}+1}^{{\frac{\un{t}}{\gm}}}(\beta\gamma+C|\Delta_{k}|^2),
 \end{equation}
where
\begin{align*}
\bar{H}_{t}^{\gamma}=\int_0^t g_\gm({\un{s}})\psg \nabla V(\bx_{{\un{s}}}),\sigma(\bx_{{\un{s}}})dB_s^H\psd=\sum_{i,j}\int_0^t g_\gm({\un{s}}) (\nabla V)_i(\bx_{{\un{s}}}),\sigma_{i,j}(\bx_{{\un{s}}})d(B_s^H)^j.
\end{align*}
with $g_\gm(s)=(1-\frac{\alpha\gm}{2})^{-\frac{s}{\gamma}}$.
We now wish to see that this relation has to be interpreted as $V(\bx_{{\un{t}}})\le e^{-\frac{\alpha({\un{t}}-{\un{s}})}{2}}V(\bx_{{\un{s}}})$, up to a remainder term.

\smallskip

\noindent
\textit{Step 2: Upper bound for $|\bar{H}_{{\un{t}}}^\gm-\bar{H}_{{\un{s}}}^{\gamma}|$.}
For every $(i,j)\in\{1,\ldots,d\}\times\{1,\ldots,q\}$, set $f^{i,j}_\gamma(s,x)=g_\gm(s) (\nabla V)_i(x)\sigma_{i,j}(x)$.
Using that $\sup_{t\in[0,T],\gamma\in(0,\gamma_0]}|g'_\gm(t)|<+\infty$, we check  that $(g_\gm(.))_{\gm\in(0,\gm_0]}$
is a family of Lipschitz continuous functions such that $\sup_{\gm\in(0,\gm_0]}[g_\gm]_{\rm Lip}<+\infty$. Furthermore, 
$(\nabla V)_i$ and $\sigma_{i,j}$ being respectively Lipschitz continuous  and  bounded Lipschitz continuous functions, we deduce that 
$(f^{i,j}_\gm)_{\gm\in(0,\gamma_0]}$ satisfies \eqref{assump:f} with $r=1$.
Applying Lemma \ref{lemme2}, we obtain that for every $\theta\in(\frac{1}{2},H)$,
$$\frac{|\bar{H}_{\un{t}}^{\gm}-\bar{H}_{\un{s}}^{\gm}|}{(\un{t}-\un{s})^{\theta}}\le C_T\left[(1+|\bx_{\un{s}}^\gm|)
+{C}_T(\un{t}-\un{s})^{\theta}(1+|\bx_{\un{s}}^\gm|^2+(\|\bar{Z}^\gm\|_{\theta,\gamma}^{\un{s},\un{t}-\gamma})^{2})\right]\|B^H\|_{\theta,T}.$$
Now, if $\un{t}-\un{s}\le \eta(\omega)$ defined by \eqref{eq:valeureta},
$$ \|\bar{Z}^\gm\|_{\theta,\gamma}^{\un{s},\un{t}-\gamma}\le \left({2}\|\sigma\|_\infty+ C_T(1+|\bx^{\gamma}_{s}|)\eta^\theta\right)\|B^H\|_{\theta,T}.$$
Owing to the definition of $\eta$, we have $a.s.$
$$ \|B^H(\omega)\|_{\theta,T}\eta^\theta\le C_T$$
where $C_T$ is a deterministic positive number so that
$$ (\|\bar{Z}^\gm\|_{\theta,\gamma}^{\un{s},\un{t}-\gamma})^{2}\le  C_T(\|B^H\|_{\theta,T}^2+1+|\bx^{\gamma}_{s}|^2).$$
Thus,
\begin{align*}
{|\bar{H}_{\un{t}}^{\gm}-\bar{H}_{\un{s}}^{\gm}|}&\le
 C_T\left[(1+|\bx_{\un{s}}^\gm|)(\un{t}-\un{s})^{\theta}+(\un{t}-\un{s})^{2\theta}(1+|\bx_{\un{s}}^\gm|^2+\|B^H\|_{\theta,T}^2)\right]
\|B^H\|_{\theta,T} .
\end{align*}
Using that $|ab|\le 2^{-1}(|a|^2+|b|^2)$ and that $1+|x|\le C\sqrt{V}(x)$,
we have
$$(1+|\bx_{\un{s}}^\gm|)(\un{t}-\un{s})^{\theta}\|B^H\|_{\theta,T}\le C(V(\bx_{\un{s}}^{\gm})(\un{t}-\un{s})^{2\theta}
+\|B^H\|_{\theta,T}^2).$$
It follows that there exists $C_T>0$ such that for every $\varepsilon>0$
\begin{equation*}
{|\bar{H}_{\un{t}}^{\gm}-\bar{H}_{\un{s}}^{\gm}|}\le \varepsilon(\un{t}-\un{s})V(\bx_{\un{s}}^\gm)\left(\frac{C_T(\un{t}-\un{s})^{2\theta-1}(1+\|B^H\|_{\theta,T})}{\varepsilon}\right)
+C_T(\|B^H\|_{\theta,T}^2+(\un{t}-\un{s})^{2\theta}\|B^H\|_{\theta,T}^3).
\end{equation*}
Now, we choose $\tilde{\eta}_\varepsilon\in(0,\eta)$
$$\frac{C_T(\tilde{\eta}_\varepsilon)^{2\theta-1}(1+\|B^H\|_{\theta,T})}{\varepsilon}\le
1.$$
More precisely, we set
$\tilde{\eta}_\varepsilon=[(C_T(1+\|B^H\|_{\theta,T}))^{-1}\varepsilon]^{\frac{1}{2\theta-1}}\wedge\eta$.
Thus, we obtain that for every
$0\le s\le t\le  T$ such that $\un{t}-\un{s}\le\tilde{\eta}_\varepsilon$,
\begin{equation}\label{eq:entretildeeta}
{|\bar{H}_{\un{t}}^{\gm}-\bar{H}_{\un{s}}^{\gm}|}\le
\varepsilon(\un{t}-\un{s})V(\bx_{\un{s}}^\gm)
+C_T(\|B^H\|_{\theta,T}^2+(\un{t}-\un{s})^{2\theta}\|B^H\|_{\theta,T}^3).
\end{equation}

\smallskip

\noindent
\textit{Step 3: Contracting dynamics for $V(\bar{X}_{\underline{k\tilde{\eta}}})$.}
Choose now $\varepsilon_0>0$ such that there exists $\delta\in(0,1/2)$ satisfying
\begin{equation}\label{eq:eee}
\forall x\in[0,1],\quad e^{-\frac{\alpha}{2}x}(1+\varepsilon_0 x)\le 1-\delta x,
\end{equation}
and set $\tilde{\eta}:=\tilde{\eta}_{\varepsilon_0}$. Plugging the two
previous controls in \eqref{eq:l41}, it follows that for every
$k\in\{1,\ldots,\lfloor \frac{T}{{\tilde{\eta}}}\rfloor\}$,
\begin{equation*}
V(\bx_{\un{k\tilde{\eta}}})\le V(\bx_{\un{{(k-1)}\tilde{\eta}}})(1-\delta
\ab_k)+C_T(1+\|B^H\|_{\theta,T}^3)+\sum_{l=\frac{\un{(k-1)\tilde{\eta}}}{\gm}+1}^{{\frac{\un{k\tilde{\eta}}}{\gamma}}}(\beta\gamma+C|\Delta_l|^2),
\end{equation*}
where $\ab_k=\un{k\tilde{\eta}}-\un{{(k-1)}\tilde{\eta}}$. Note that we can
apply \eqref{eq:eee} since
$$\ab_k\le 2\tilde{\eta}\le 2\eta\le 2^{1-\frac{1}{\theta}}\le 1.$$
In particular, $\delta\ab_k\le 1/2$. With the convention $\prod_{\emptyset}=1$,
an iteration of this inequality yields for every $k\in\{1,\ldots,\lfloor
\frac{T}{\tilde{\eta}}\rfloor\}$:
\begin{align*}
V(\bx_{\un{k\tilde{\eta}}})&\le V(x)\prod_{l=1}^k(1-\delta
\ab_l)+C_T\|B^H\|_{\theta,T}^3\sum_{m=1}^k\prod_{l=m+1}^k(1-\delta \ab_l)
\\
&+\sum_{m=1}^k\prod_{l=m+1}^k(1-\delta \ab_l)
\sum_{l=\frac{\un{(m-1)\tilde{\eta}}}{\gm}+1}^{{\frac{\un{m\tilde{\eta}}}{\gamma}}}(\beta\gamma+C|\Delta_l|^2).
\end{align*}

Then, using the inequality $\log(1+x)\le x$ $\forall x\in(-1,+\infty)$, we have
for every $m\in\{0,\ldots,k\}$ (with the convention $\sum_\emptyset=0$)
$$\prod_{l=m+1}^k(1-\delta \ab_l)=\exp(\sum_{l=m+1}^k\log(1-\delta\ab_l))\le
\exp(-\sum_{l=m+1}^k\delta\ab_l))=\exp(-\delta
\un{k\tilde{\eta}}+\delta\un{m\tilde{\eta}}).$$
Thus
$$\sum_{m=1}^k\prod_{l=m+1}^k(1-\delta \ab_l)\le 
\exp(-\delta\un{k\tilde{\eta}})\sum_{m=1}^k\exp(\delta\tilde{\eta})^m\le 
\exp(\delta\tilde{\eta}-\delta\un{k\tilde{\eta}})\frac{\exp(k\delta\tilde{\eta})-1}{\exp(\delta\tilde{\eta})-1}\le
\frac{C}{\delta\tilde{\eta}}$$
where $C$ is deterministic (and does not depend on $k$). Owing to the definition
of $\tilde{\eta}$ (and thus from that of $\eta$), we have 
$${\tilde{\eta}}^{-1}\le[(C_T(1+\|B^H\|_{\theta,T}){)}^{-1}\varepsilon_0]^{-\frac{1}{2\theta-1}} \lor \eta^{-1}\le
{C}_{\varepsilon_0,T} (1+\|B^H\|^\frac{1}{2\theta-1}). $$
It follows that there exists a polynomial function $P_1$ such that
$$C_T\|B^H\|_{\theta,T}^3\sum_{m=1}^k\prod_{l=m+1}^k(1-\delta \ab_l)\le
P_1(\|B^H\|_{\theta,T}).$$
On the other hand, since $\prod_{l=m+1}^k(1-\delta \ab_l)\le 1$, we also have
\begin{align*}
\sum_{m=1}^k\prod_{l=m+1}^k(1-\delta
\ab_l)\sum_{l=\frac{\un{(m-1)\tilde{\eta}}}{\gm}+1}^{{\frac{\un{m\tilde{\eta}}}{\gamma}}}(\beta\gamma+C|\Delta_l|^2)
\le \sum_{u=1}^{\lfloor
\frac{\un{k\tilde{\eta}}}{\gamma}\rfloor}(\beta\gamma+C|\Delta_u|^2)\le \beta
k\tilde{\eta}+C\sum_{u=1}^{\lfloor
\frac{\un{k\tilde{\eta}}}{\gamma}\rfloor}|\Delta_u|^2.
\end{align*}
We deduce that for every
$k\in\{1,\ldots,\lfloor \frac{T}{{\tilde{\eta}}}\rfloor \}$:
\begin{align}\label{eq:controlptk1}
V(\bx_{\un{k\tilde{\eta}}})\le
V(x)\exp(-\delta\un{k\tilde{\eta}})+P_1(\|B^H\|_{\theta,T})
+C Q_\gamma( B^H_t,0\le t\le T),
\end{align}
where $P_1$ is a polynomial function and $Q_\gamma$ is defined by
\begin{equation}\label{eq:qgamma}
Q_\gamma( (w(t))_{t\in[0,T]})=\sum_{k=1}^{\lfloor
\frac{T}{\gamma}\rfloor}|w(k\gamma)-w((k-1)\gamma)|^2.
\end{equation}
Owing to the definition of $\|B^H\|_{\theta,T}$, one checks that for every
$\gamma\in(0,\gamma_0]$
$$ Q_\gamma( B^H_t,{t\in[0,T]})\le \gamma^{2\theta-1} T
\|B^H\|_{\theta,T}^{2}\le C_T \|B^H\|_{\theta,T}^{2}.$$
Thus, denoting by $P$ the polynomial function defined by $P(v)=P_1(v)+C_T v^2$,
 we deduce from \eqref{eq:controlptk1} that for every
$k\in\{1,\ldots,\lfloor \frac{T}{\tilde{\eta}}\rfloor \}$:
 \begin{align}\label{eq:controlptk}
V(\bx_{\un{k\tilde{\eta}}})\le
V(x)\exp(-\delta\un{k\tilde{\eta}})+P(\|B^H\|_{\theta,T}).
\end{align}

\smallskip

\noindent
\textit{Step 4: Contracting dynamics for $V(\bar{X}_{T})$.}
We now patch the estimates obtained so far in order to propagate inequality \eqref{eq:controlptk} to $V(\bar{X}_{T})$. Indeed, applying \eqref{eq:controlptk} with $k=\lfloor {\tilde{\eta}}^{-1}{T}\rfloor$, we obtain
\begin{align*}
V(\bx_{\un{\lfloor {\tilde{\eta}}^{-1}{T}\rfloor\tilde{\eta}}})\le V(x)\exp(-\delta{\un{\lfloor {\tilde{\eta}}^{-1}{T}\rfloor\tilde{\eta}}})+P(\|B^H\|_{\theta,T}),
\end{align*}
and owing again to \eqref{eq:l41}, \eqref{eq:entretildeeta} (applied with $s=\lfloor {\tilde{\eta}}^{-1}{T}\rfloor\tilde{\eta}$ and $t=T$) and \eqref{eq:eee}, we deduce that
\begin{align}\label{l42}
V(\bx_{\un{T}})\le V(x)\exp(-\delta{\un{\lfloor {\tilde{\eta}}^{-1}{T}\rfloor\tilde{\eta}}})+\tilde{P}(\|B^H\|_{\theta,T})
\end{align}
where $\tilde{P}$ is a polynomial function.
Finally, we want to control $V(\bar{X}_T)-V(\bar{X}_{\un{T}})$. The function
$\nabla V$ being sublinear and $D^2V$ being bounded, we deduce from the Taylor
formula that for every $x,y\in\ER^d$,
$$V(y)\le V(x)+C(|x|.|y-x|+|y-x|^2).$$
Applying this inequality with $x=\bar{X}_{\un{T}}$ and $y=\bar{X}_T$ and taking
advantage of the assumptions on $b$, we have
\begin{align}\label{eq:detbarat}
V(\bx_T)&\le V(\bx_{\un{T}})+C\left[\gm(1+ |\bx_{\un{T}}|^2)+(1+|\bx_{\un{T}}|)
|B^H_T-B^H_{\un{T}}|+|B^H_T-B^H_{\un{T}}|^2\right]\\
&\le V(\bx_{\un{T}})(1+C\gm)+C(1+\|B^H\|_{\theta,T}^2),
\end{align}
where in the second line, we again used the elementary inequality
$|ab|\le2^{-1}(|a|^2+|b|^2)$ and the
fact that $|x|^2\le C V(x)$.
Combined with \eqref{l42}, the previous inequality yields:
\begin{align*}
V(\bx_T)\le V(x)\exp(-\delta {\un{\lfloor {\tilde{\eta}}^{-1}{T}\rfloor\tilde{\eta}}})(1+C\gamma)+P_{1,\theta}(\|B^H\|_{\theta,T}),
\end{align*}
where $P_{1,\theta}$ denotes the polynomial function defined by $P_{1,\theta}(v)=\tilde{P}(v)+C(1+v^2$). Finally, since $\exp(-\delta {\un{\lfloor {\tilde{\eta}}^{-1}{T}\rfloor\tilde{\eta}}})\le e^{-\delta (T-\tilde{\eta}-\gamma)}$, since $T\ge1$ and $\tilde{\eta}\le 2^{\frac{1}{\theta}}<1$, one can find  $\gm_0>0$ such that $T-\delta\tilde{\eta}-\gamma_0>0$ and such that,
$$\exp(-\delta {\un{\lfloor {\tilde{\eta}}^{-1}{T}\rfloor\tilde{\eta}}})(1+C\gamma)\le \rho\quad a.s.$$
Inequality \eqref{VX}  {for $p=1$} follows.\\

\smallskip

\noindent
\textit{Step 5: Inequality \eqref{VX} for $p>1$.} We recall that for every $p>0$, there exists $c_p>0$ such that for every  $u,v\in\ER$, the following  inequality holds:   $|u+v|^p\le |u|^p+c_p(|v|.| u|^{p-1}+|v|^p)$. Thus, by the Young inequality, it follows that for every $\varepsilon>0$, there exists $c_{\varepsilon,p}>0$ such that $|u+v|^p\le (1+\varepsilon)|u|^p+c_{\varepsilon,p}|v|^p$ for every $u,v \in \ER$ and $p\ge1$. Applying this inequality,  we deduce from the case $p=1$  that
$$V^p(\bx^\gm_T)\le \rho^p(1+\varepsilon) V^p(x)+C_TP_{\theta}(\|B^H\|_{\theta,1}))^p.$$
Since $\rho<1$, we can choose $\varepsilon>0$ such that $\tilde{\rho}=\rho^p(1+\varepsilon)<1$. It follows that 
$$V^p(\bx^\gm_T)\le \tilde{\rho} V^p(x)+P_{p,\theta}(\|B^H\|_{\theta,T})$$
where $P_{p,\theta}$ is again a polynomial function.\\
Now, let us focus on  \eqref{eq:sup-V}. We only give the main ideas of the proof when $p=1$ (the extension to $p>1$ again follows from the inequality $|u+v|^p\le 2^{p-1}(|u|^p+|v|^p)$). By  \eqref{eq:controlptk}, 
the announced inequality holds taking the supremum of the left-hand side of \eqref{eq:sup-V} for every 
$\un{k\tilde{\eta}}$ with $k\in \{1,\ldots,\lfloor \frac{T}{\tilde{\eta}}\rfloor\}$. Then, for every $t\in[\un{(k-1)\tilde{\eta}},\un{k\tilde{\eta}}]$, it remains to control (uniformly in $k$) $V(\bar{X}_t)$ in terms of $V(\bar{X}_{\un{(k-1)\tilde{\eta}}})$.   By \eqref{entretildeeta2} and \eqref{eq:entretildeeta}, we obtain such a control for every discretization time between $\un{(k-1)\tilde{\eta}}$ and $\un{k\tilde{\eta}}$.  Then, it is enough to control uniformly $V(\bar{X}_t)$ in terms of $V(\bar{X}_{\un{t}})$. This can be done similarly as in inequality  \eqref{eq:detbarat}.

\smallskip

\noindent
\textit{Step 6: Proof of the H\"older bound \eqref{eq:holder-bnd-with-V}.} Let $s,t\in[0,T]$ with $0\le s<t\le T$. We have
$$\bx^\gm_t-\bx^\gm_s=\int_s^t b(\bx^\gm_{\un{u}})du+\bar{Z}^\gm_t-\bar{Z}^\gm_s.$$
First, since $|b(x)|\le C\sqrt{V}(x)\le C(1+V(x))$,
$$|\int_s^t b(\bx^\gm_{\un{u}})du|\le C (t-s)(1+\sup_{u\in[0,T]} V(\bx_{\un{u}}))$$
and it follows from $(i)$ that 
$$\sup_{0\le s<t\le T}\frac{|\int_s^t b(\bx^\gm_{\un{u}})du|}{(t-s)^\theta}\le C_T (V(x)+{P}_{p,\theta}(\|B^H\|_{\theta,T})).$$
Thus, we can only focus on the increment of $\bar{Z}^\gm$. By Lemma \ref{lemme3}, for every $u,v\in[0,T]$ such that $\un{v}-\un{u}\le \eta$ (where $\eta$ is given by \eqref{eq:valeureta}),
$$ {|\bar{Z}^\gamma_{\un{v}}-\bar{Z}^\gamma_{\un{u}}|}\le  {(\un{v}-\un{u})^\theta}\left(2\|\sigma\|_\infty+ C_T(1+\sup_{s\in[0,T]}|\bx_{\un{s}}|)\eta^\theta\right)\|B^H\|_{\theta,T}.$$
Using the concavity of $x\mapsto x^\theta$ on $\ER_+$, we have for every $s_1,s_2\in[0,T]$ being such that $|s_2-s_1|\le \gamma$,
$$|\bar{Z}^\gm_{s_2}-\bar{Z}^\gm_{s_1}|\le   2^{1-\theta}\|\sigma\|_\infty (s_2-s_1)^\theta \|B^H\|_{\theta,T}$$
and   we derive that for every $u,v\in[0,T]$ with $|u-v|\le\eta$,
$$ {|\bar{Z}^\gamma_{{v}}-\bar{Z}^\gamma_{{u}}|}\le  C_T{({v}-{u})^\theta}\left(1+ (1+\sup_{s\in[0,T]}|\bx_{\un{s}}|)\eta^\theta\right)\|B^H\|_{\theta,T}.$$
Now, by the very definition of $\eta$, we have $\eta^\theta\|B^H\|_{\theta,T}\le 1$. Then, since $|x|^2\le C V(x)$, we have in particular that $|x|\le C V(x)$ (using that $\inf_{x\in\ER^d} V(x)>0$) and  we deduce from the first part of this proposition that for every $u,v\in[0,T]$ with $|u-v|\le\eta$:
\begin{equation}\label{eq:lemme3}
 {|\bar{Z}^\gamma_{{v}}-\bar{Z}^\gamma_{{u}}|}\le  C_T({v}-{u})^\theta  (V(x)+\tilde{P}(\|B^H\|_{\theta,T})),
 \end{equation}
where $\tilde{P}$ is a polynomial function.\\
We want now to make use of the previous inequality to control $\bar{Z}^\gm_t-\bar{Z}^\gm_s$ for every  $0\le s<t\le T$.  We divide  $[s,t]$ in intervals of length lower than $\eta$. More precisely,
setting $s_k=s+ k\lfloor \eta\rfloor$, we have  
$$\bar{Z}^\gm_t-\bar{Z}^\gm_s=\bar{Z}^\gm_t-\bar{Z}^\gm_{s_{\lfloor\frac{t-s}{\eta}\rfloor}}+\sum_{k=1}^{\lfloor\frac{t-s}{\eta}\rfloor} \bar{Z}^\gm_{s_{k}}-\bar{Z}^\gm_{s_{k-1}}.$$
Then, we deduce from \eqref{eq:lemme3} that
\begin{align*}
 |\bar{Z}^\gm_t-\bar{Z}^\gm_s|&\le  C_T\left((t-s_{\lfloor\frac{t-s}{\eta}\rfloor})^\theta +\lfloor\frac{t-s}{\eta}\rfloor\eta^\theta\right)
  (V(x)+\tilde{P}(\|B^H\|_{\theta,T}))\\
  &\le C_T\left((t-s)^\theta +(t-s)\eta^{\theta-1}\right)
  (V(x)+\tilde{P}(\|B^H\|_{\theta,T})).
\end{align*}
Thus, using \eqref{eq:lemme3} if $t-s\le\eta$ or the fact that $(t-s)\eta^{\theta-1}\le (t-s)^\theta$ if $t-s\ge \eta$, we deduce that there exists $C_T>0$ such that for every 
$0\le s< t\le T$,
\begin{equation*}
 {|\bar{Z}^\gamma_{{t}}-\bar{Z}^\gamma_{{s}}|}\le  C_T({t}-{s})^\theta  (V(x)+\tilde{P}(\|B^H\|_{\theta,T})).
 \end{equation*} 
The result \eqref{eq:holder-bnd-with-V} follows.

\section{Tightness properties}\label{section4}
In the following proposition, we obtain some $a.s.$ tightness results for the sequence $({\cal P}^{(n,\gamma)}(\omega,d\alpha))_{n\ge1}$. Using that the controls established in Proposition \ref{lemme4} are uniform in $\gamma$, we also show that tightness properties also hold for the set of its limiting measures $({\cal U}^{(\infty,\gamma)}(\omega,\theta))_\gamma$  defined by
$${\cal U}^{(\infty,\gamma)}(\omega,\theta)=\left\{\mu\in\bar{\cal C}^\theta(\ER_+,\ER^d),\exists (n_k(\omega))_{k\ge1},{\cal P}^{(n_k(\omega),\gamma)}(\omega,d\alpha)\xrn{k\nrn}\mu\right\}.$$ 
\begin{prop}\label{prop:ntendinfty}
Assume $\mathbf{(C)}$. Then, there exists $\gamma_0>0$ such that,\\
(i) For every $\gamma\in(0,\gamma_0]$ and $p\ge1$, $a.s.$,
$$\limsup_{n\rightarrow+\infty}\frac{1}{n}\sum_{k=1}^{n} V^p(\bx^\gm_{\gamma(k-1)})\le
C_p\ES[|P_{p,\theta}(\|B^H\|_{\theta,1})|]<+\infty.$$
where $C_p$ does not depend on $\gm$ and $P_{p,\theta}$ is a polynomial function.\\
(ii) For every $\theta\in(1/2,H)$, for every $\gamma\in(0,\gamma_0]$, $({\cal P}^{(n,\gamma)}(\omega,d\alpha))_{n\ge1}$ is almost surely tight on $\bar{\cal C}^\theta(\ER_+,\ER^d)$.\\
(iii) For every $\theta\in(1/2,H)$, $({\cal U}^{(\infty,\gamma)}(\omega,\theta))_{\gamma\in(0,\gamma_0]}$ is $a.s.$ tight in $\bar{\cal C}^\theta(\ER_+,\ER^d)$.
\end{prop}
\begin{proof}
\textit{(i)} 
\textbf{Case $p=1$} :  We first focus on the sequence $(\frac{1}{N}\sum_{\ell=0}^{N-1} V(\bx^\gm_\ell))_{N\ge1}$. Note that, at this stage, we consider the values of the Euler scheme at {times} $0$, $1$, $2$,~\ldots (which do not depend on $\gamma$). that  We set 
$$\forall \ell\ge 0,\quad(\delta_\ell B^H)_t=B^H_{\ell+t}-B^H_{\ell}.$$ By Proposition \ref{lemme4} applied with $T=1$, we have for every $k\ge1$
$$V(\bx^\gm_{\ell})\le \rho V(\bx^\gm_{\ell-1})+P_{1,\theta}(\|\delta_{\ell-1}B^H\|_{\theta,1})$$
with $\rho\in(0,1)$. An iteration yields for every $\ell\ge1$
$$V(\bx^\gm_{\ell})\le \rho^\ell V(x)+\sum_{m=0}^{\ell-1}\rho^{\ell-1-m}P_{1,\theta}(\|\delta_{m}B^H\|_{\theta,1}).$$
Setting $U_m=P_{1,\theta}(\|\delta_{m}B^H\|_{\theta,1})$ and summing over $\ell$, we obtain
\begin{align*}
\frac{1}{N}\sum_{\ell=0}^{N-1} V(\bx^\gm_{\ell})&\le \frac{V(x)}{N(1-\rho)}+\frac{1}{N}\sum_{\ell=0}^{N-1}\sum_{m=0}^{\ell-1}\rho^{\ell-1-m}U_m\\
&\le \frac{V(x)}{N(1-\rho)}+\frac{1}{N}\sum_{m=0}^{N-2}U_m\sum_{\ell=m+1}^N\rho^{\ell-1-m}
\le \frac{V(x)}{N(1-\rho)}+\frac{1}{N(1-\rho)}\sum_{m=0}^{N-2}U_m.
\end{align*}

Let us remark that since $ B^H$ is a   $\bar{\cal C}^\theta([0,1],\ER^q) $  valued Gaussian random variable, the 
norm $ \|B^H\|_{\theta,1} $ has finite moments  of every order, which is classical consequence 
of Fernique Lemma. 
Hence 
\begin{equation}\label{eq:erg}
\ES[|P_{1,\theta}(\|B^H\|_{\theta,1})|]<+\infty.
\end{equation}
Then, since $(\delta_{m}B^H)_{m\ge1}$ is ergodic (see Remark \ref{maruyamaremark} for background and details). 
We have
\begin{equation}\label{p11}
\frac{1}{N}\sum_{m=0}^{N-2}U_m\xrn{N\nrn}\ES[P_{1,\theta}(\|B^H\|_{\theta,1})]\quad a.s.
\end{equation}
and it follows that 
\begin{equation}\label{p12}
\limsup_{N\rightarrow+\infty}\frac{1}{N}\sum_{\ell=0}^{N-1} V(\bx^\gm_{\ell})\le \frac{1}{1-\rho}\ES[P_{1,\theta}(\|B^H\|_{\theta,1})]\quad a.s.
\end{equation}
We want now to use this result to control the $a.s.$ asymptotic behavior of $(\frac{1}{n}\sum_{k=0}^{n-1} V(\bx^\gm_{\gamma k}))_{n\ge1}$. 
By the second point of Proposition \ref{lemme4}(i), for every $\ell\ge0$,
$$\sup_{k\in[\lfloor \frac{\ell}{\gamma}\rfloor+1,\lfloor \frac{\ell+1}{\gamma}\rfloor]} V(\bx^\gm_{\gm k})\le C \left(V(\bx^\gm_{\ell})+P_{1,\theta}(\|\delta_{\ell}B^H\|_{\theta,1})\right).$$
As a consequence, setting $N=\lfloor \gamma (n-1)\rfloor+1$, we have
\begin{align*}
\frac{1}{n}\sum_{k=0}^{n-1} V(\bx^\gm_{\gamma k})&\le \frac{N}{n}\frac{1}{N}\left(V(x)+\sum_{\ell=0}^{N-1}\sum_{k=\lfloor \frac{\ell}{\gamma}\rfloor+1}^{\lfloor \frac{\ell+1}{\gamma}\rfloor} V(\bx^\gm_{\gm k})\right)\\
&\le C(\gamma+\frac{1}{n})(\frac{1}{\gm}+1)\left(\frac{1}{N}\sum_{\ell=0}^{N-1} \left(V(\bx^\gm_{\ell})+P_{1,\theta}(\|\delta_{\ell}B^H\|_{\theta,1})\right)\right).
\end{align*}
Using \eqref{p11} and \eqref{p12}, the result follows when $p=1$.

The proof when $ p > 1 $ is  very similar to the case $p=1$ and is left to the reader.\\

\textit{(ii)} If for a sequence $(\mu_n)_{n\ge1}$ of probability measures on $\ER^d$, there exists a positive function $\varphi:\ER^d\mapsto (0,+\infty)$
such that $\sup_{n\ge1}\mu_n(\varphi)<+\infty $ and $\lim_{|x|\rightarrow+\infty}\varphi(x)=+\infty$,
one classically derives that $(\mu_n)_{n\ge1}$ is tight on $\ER^d$ (see $e.g.$ \cite{duflo} p. 41). Thus, by $\textit{(i)}$, $({\cal P}^{(n,\gamma)}_0(\omega,dx))$ is $a.s.$ tight on $\ER^d$.   Owing to some classical tightness results in Hölder spaces (see $e.g.$ \cite{suquet}, Theorem 1.4), we deduce that we only have  to prove that for every $T>0$, for every $\theta\in(1/2,H)$, for every $\varepsilon>0$,
\begin{equation}\label{omegaTT}
\limsup_{\delta\rightarrow0}\limsup_{n\rightarrow+\infty}\frac{1}{n}\sum_{k=1}^n {\bf 1}_{\{\omega_{\theta,T}(\bar{X}^\gamma_{\gamma(k-1)+.},\delta)\ge\varepsilon\}}=0,
\end{equation}
where we recall that 
$$\forall\, T>0,\quad \omega_{\theta,T}(f,\delta):=\sup_{0\le s<t<T,0\le |t-s|\le\delta}\frac{|f(t)-f(s)|}{|t-s|^\theta}.$$
By Proposition \ref{lemme4} \emph{(ii)} with $\theta'\in(\theta,H)$,
$$\sup_{0\le s<t\le T}\frac{|\bx^\gm_t-\bx^\gm_s|}{(t-s)^\theta}\le C_T (t-s)^{\theta'-\theta}(V(x)+\tilde{P}_{\theta'}(\|B^H\|_{\theta',T}))$$
so that for every $s,t\in[0,T]$ such that $s<t$ and $t-s\le\delta$,
 $$\sup_{0\le s<t\le T}\frac{|\bx^\gm_t-\bx^\gm_s|}{(t-s)^\theta}\le C_T \delta^{\theta'-\theta}(V(x)+\tilde{P}_{\theta'}(\|B^H\|_{\theta',T})).$$
As in $(i)$, this property can be extended to the shifted process: we have for every $k\ge0$
\begin{equation}\label{ghtrz}
\omega_{\theta,T}(\bar{X}^\gamma_{\gamma k+.},\delta)=\sup_{0\le s<t\le T, \; t-s \le \delta}\frac{|\bx^\gm_{\gamma k+t}-\bx^\gm_{\gm k+s}|}{(t-s)^\theta}\le C_T \delta^{\theta'-\theta}(V(\bar{X}^\gm_{\gm k})+\tilde{P}_{\theta'}(\|\delta_{ k}B^H\|_{\theta',T}).
\end{equation}

Since $(\delta_k B^H)_{k\ge1}$ is ergodic (see Remark \ref{maruyamaremark} for details) and since by the Fernique Lemma  $\| B^H\|_{\theta',T}$ has moments of any order, we have 
$$\frac{1}{n}\sum_{k=1}^n\tilde{P}_{\theta'}(\|\delta_{ k}B^H\|_{\theta',T})\xrn{n\nrn} \ES[\tilde{P}_{\theta'}(\| B^H\|_{\theta',T})]\quad a.s.$$
Then, we deduce from $(i)$ and \eqref{ghtrz} that
$$\limsup_{n\rightarrow+\infty}\frac{1}{n}\sum_{k=1}^n {\omega_{\theta,T}(\bar{X}^\gamma_{\gamma(k-1)},\delta)}\le C\delta^{\theta'-\theta}.$$
By the Markov inequality, we obtain for every $\varepsilon>0$,
\begin{equation}\label{eq:prop221}
\limsup_{n\rightarrow+\infty}\frac{1}{n}\sum_{k=1}^n {\bf 1}_{\{\omega_{\theta,T}(\bar{X}^\gamma_{\gamma(k-1)+.},\delta)\ge\varepsilon\}}\le C\frac{\delta^{\theta'-\theta}}{\varepsilon}
\end{equation} 
and \eqref{omegaTT} follows.

%

\textit{(iii)} Let $\theta\in(1/2,H)$ and  denote by $\mu^{(\gamma)}$ an element of ${\cal U}^{(\infty,\gamma)}(\omega,\theta) $ and by 
$\mu^{(\gamma)}_t$ its marginals. By \eqref{eq:erg} and \eqref{p12},  
$$\forall \gamma\in(0,\gm_0],\quad \mu^{(\gamma)}_0(V)\le \frac{C}{1-\rho}$$
where $\rho$ does not depend on $\gm$. It follows that ${\cal U}_0^{(\infty,\gamma)}(\omega,\theta) $ is $a.s.$ tight in $\ER^d$ (where ${\cal U}_0^{(\infty,\gamma)}(\omega,\theta) $ stands for the set of initial distributions $\mu^{(\gamma)}_0$).\\
Now, since $C$ does not depend on $\gm$ in \eqref{eq:prop221}, we also have for every $T>0$, $\delta>0$ and $\varepsilon>0$
 for every $\theta'>\theta$:
$$\forall \gamma\in(0,\gm_0],\quad \mu^{(\gamma)}({\bf 1}_{\{\omega_{\theta,T}(.,\delta)\ge\varepsilon\}})\le C\delta^{\theta'-\theta}$$
and the announced result follows again from Theorem 1.4 of  \cite{suquet}.
\end{proof}

\begin{Remarque}\label{maruyamaremark}
Some of the arguments of the previous proof are based on the ergodicity of the increments of the fractional Brownian motion.
More precisely, we use the fact that $(B_t^H)_{t\in\ER}$ is ergodic under the transformation  
$T_\xi:\bar{\cal C}^\theta(\ER,\ER^q)\rightarrow\bar{\cal C}^\theta(\ER,\ER^q) $ defined by
$(T_\xi(\omega))_t=\omega(\xi+t)-\omega(\xi)$ $(\xi>0)$, which implies by the Birkhoff theorem that, for any functional $F:\bar{\cal C}^\theta(\ER,\ER^q)\rightarrow\ER$
such that $\ES[|F(B^H_t,t\ge0)|]<+\infty$,
\begin{equation}\label{ergodicityfrac}
\PE-a.s.,\quad\frac{1}{n}\sum_{k=1}^n F(B^H_{\xi k+.}-B^H_{\xi k})\xrightarrow{n\rightarrow+\infty}\ES[F(B^H_t,t\ge0)].
\end{equation}
Note that this ergodic  result is a (classical) consequence of the Maruyama theorem \cite{maruyama} (see also \cite{weber}) which is stated in a slightly different way: let $({\theta}_t)_{t\in\ER}$ denote the standard time-shift defined for $\omega:\ER\rightarrow\ER$ by 
$\theta_t(\omega)=\omega(t+.)$. Then, a centered stationary real Gaussian process $(Y_t)_{t\in\ER}$ is ergodic under $({\theta}_t)_{t\in\ER}$ if its covariance function $r(t)=\ES[Y_tY_0]$ satisfies $r(t)\rightarrow0$ as $t\rightarrow+\infty$. 
This result can be applied to  the stationary (centered) fractional Ornstein-Uhlenbeck process solution to $dY_t=-Y_t dt+ dB_t^H$
(since $r(t)\rightarrow0$, see $e.g.$ \cite{cheridito})). Then  we retrieve \eqref{ergodicityfrac} by using that the increment $B^H_{t+s}-B^H_t$ is a functional of $(Y_t)_{t\ge s}$:
$B^H_{t+s}-B^H_t=Y_{t+s}-Y_t+\int_s^t Y_u du.$
\end{Remarque}
\section{Identification of the weak limits}\label{section5}
\subsection{Weak limits of $({\cal P}^{(n,\gamma)}(\omega,d\alpha))_{n\ge1}$}
\label{identification1}
We have the following result:
\begin{prop}\label{prop2} Assume $\mathbf{(C)}$ and let ${{\cal P}}^{(\infty,\gamma)}(\omega,d\alpha)$ denote a weak limit of $({{\cal P}}^{(n,\gamma)}(\omega,d\alpha))_{n\ge1}$. Then, ${{\cal P}}^{(\infty,\gamma)}(\omega,d\alpha)$ is $a.s.$ an adapted stationary solution 
of \eqref{fractionalSDE0-disc}.
\end{prop}
\begin{Remarque} In the following proof, we will state  some properties ``for every function $f,$ for almost every $\omega$'' and conclude that ``for almost every $\omega,$ for every function $f$'' the property is true. For the sake of completeness, we recall here that such inversions are rigorous  since we work on Polish spaces (in which the distributions and the weak convergence are characterized by some countable family of bounded continuous functions). 
\end{Remarque}
\begin{proof}
In the proof, we denote by $(\tilde{{\cal P}}^{(n)}(\omega,d\alpha, d\beta))_{n\ge 1}$,  the sequence of probability measures on $\bar{{\mathcal C}}^\theta(\ER_+,\ER^d) \times \bar{{\mathcal C}}^\theta(\ER,\ER^q) $ 
with $ \frac12 < \theta<H$   defined by
 $$\tilde{{\cal P}}^{(n,\gamma)}(\omega,d\alpha, d\beta)=\frac{1}{n}\sum_{k=1}^n
 {\delta}_{({\bar{X}}^\gamma_{\gamma(k-1)+.}(\omega), B^H_{(k-1)\gamma +.}(\omega) - B^H_{(k-1)\gamma}(\omega) )}
(d\alpha, d\beta)$$ where 
 $ (B^H_t)_{t\in\ER} $ is the fractional Brownian motion used to build the Euler scheme~\eqref{eq:euler-scheme-x}.
 First, let us recall that by Proposition~\ref{prop:ntendinfty} \emph{(ii)}, $({\cal P}^{(n,\gamma)}(\omega,d\alpha))_{n\ge1}$ is $a.s.$ tight. Thus, we can consider a weak limit  ${{\cal P}}^{(\infty,\gamma)}(\omega,d\alpha)$. Second, one checks that
  $(\tilde{{\cal P}}^{(n,\gamma)}(\omega,d\alpha, d\beta))_{n\ge1}$ is also almost surely tight since each of its margins have this property. Indeed, for the first margin, it is again \emph{(ii)} of Proposition~\ref{prop:ntendinfty}. 
For the second margin, we use that $(B^H_t)_{t\in\ER}$ is ergodic under the transformation  
$T_\gamma:\bar{\cal C}^\theta(\ER,\ER^q)\rightarrow\bar{\cal C}^\theta(\ER,\ER^q) $ (see Remark \ref{maruyamaremark}).  In particular,
\begin{equation}\label{eq:ergoinc}
 \frac{1}{n}\sum_{k=1}^n  {\delta}_{B^H_{(k-1)\gamma +.} - B^H_{(k-1)\gamma}} (d\beta)
 \end{equation}
is converging almost surely to the distribution of $(B^H_t)_{t\in\ER}$ (on $\bar{\cal C}^\theta(\ER,\ER^q)$).
Hence,  the sequence $(\tilde{{\cal P}}^{(n)}(\omega,d\alpha, d\beta))_{n\ge 1}$ is almost surely tight (and thus relatively compact).
Then, if  ${{\cal P}}^{(\infty,\gamma)}(\omega,d\alpha)$ is the limit of a subsequence of $({\cal P}^{(n,\gamma)}(\omega,d\alpha))_{n\ge1}$, maybe  with the help  of a second extraction, it follows  that $a.s.$, there exists a subsequence $(n_k(\omega))_{k\ge0}$ such that
\begin{equation}\label{eq:convptilde}
 {\cal P}^{(n_k,\gamma)}(\omega,d\alpha)\xrn{k\nrn}{\cal P}^{(\infty,\gamma)}(\omega,d\alpha)\quad\textnormal{and}\quad
\tilde{{\cal P}}^{(n_k,\gamma)}(\omega,d\alpha,d\beta)\xrn{n_k\nrn}\tilde{{\cal P}}^{(\infty,\gamma)}(\omega,d\alpha,d\beta)
\end{equation}
where the first margin of $\tilde{{\cal P}}^{(\infty,\gamma)}(\omega,d\alpha,d\beta)$ is obviously ${\cal P}^{(\infty,\gamma)}(\omega,d\alpha)$ and the second one is $a.s.$ the distribution of
$(B_t^H)_{t\in\ER}$ (thanks to \eqref{eq:ergoinc}). Let us also denote by  $(X^{(\infty,\gamma)}_t, B^H_t) $ the coordinate process on $\bar{{\mathcal C}}^\theta(\ER_+,\ER^d) \times \bar{{\mathcal C}}^\theta(\ER,\ER^q) $ 
endowed with the probability $\tilde{{\cal P}}^{(\infty,\gamma)}.$
For $(\alpha,\beta) \in  \bar{{\mathcal C}}^\theta(\ER_+,\ER^d) \times \bar{{\mathcal C}}^\theta(\ER_+,\ER^q) $ we consider the following function 
\begin{equation}\label{eqcvsub}
 \tilde{\Phi}^{\gamma}(\alpha,\beta)_t := \alpha_0 + \int_0^t b(\tilde{\Phi}^{\gamma}(\alpha,\beta)_{\un{s}_{\gamma}}) ds +   \int_0^t  \sigma(\tilde{\Phi}^{\gamma}(\alpha,\beta)_{\un{s}_{\gamma}}) d \beta_s.
 \end{equation}
Please remark that $ \tilde{\Phi}^{\gamma} $ is slightly different from $ \Phi^{\gamma} $ in the way it handles 
the initial condition but 
$$ \tilde{\Phi}^{\gamma}(\alpha,\beta) = \Phi^{\gamma}(a,\beta)$$
for every $\alpha$ such that $ \alpha_0=a.$ 
For $t,\; K > 0 $ let us denote by $F_{t,K}$ the functional defined on $  \bar{{\mathcal C}}^\theta(\ER_+,\ER^d) \times \bar{{\mathcal C}}^\theta(\ER,\ER^q)$
by  $ F_{t,K}(\alpha,\beta)= \sup_{ 0 \le s \le t} |\alpha_s -  \tilde{\Phi}^{\gamma}(\alpha,\beta_+)_s | \wedge K $ 
where $\beta_+=(\beta(t))_{t\ge0}$. The function $F_{t,K}$ is bounded continuous on $  \bar{{\mathcal C}}^\theta(\ER_+,\ER^d) \times \bar{{\mathcal C}}^\theta(\ER,\ER^q).$

\noindent Then,
$$ \ES ( F_{t,K}(X^{(\infty,\gamma)}, B^H))= \lim_{n_l \to \infty} \frac{1}{n_l}\sum_{k=1}^{n_l} F_{t,K}({\bar{X}}_{(k-1)\gamma+.}^\gm,B^H_{(k-1)\gamma +.} - B^H_{(k-1)\gamma} ).$$
By definition of the Euler scheme (even though it is shifted), we have for every $k\ge1$, $ F_{t,K}({\bar{X}}_{(k-1)\gamma+.}^\gamma,B^H_{(k-1)\gamma +.} - B^H_{(k-1)\gamma} )= 0 $
almost surely, and 
$$ X^{(\infty,\gamma)}= \tilde{\Phi}^{\gamma}(X^{(\infty,\gamma)},B^H)$$
almost surely, which ensures that the pair $ (X^{(\infty,\gamma)},B^H) $ is a solution of~\eqref{fractionalSDE0-disc}. \\
The stationarity of  $X^{(\infty,\gamma)}$ follows from the construction. Actually, using the convergence of $({\cal P}^{(n,\gamma)}(\omega,d\alpha))$, we have for every bounded continuous  functional $F:\bar{{\cal C}}^\theta(\ER_+,\ER^d)\rightarrow\ER$,
$$\frac{1}{n}\sum_{k=1}^n F(\bx^\gm_{\gm(k-1)+t+.})-F(\bx^\gm_{\gm(k-1)+.})\xrn{n\nrn}{\ES}[F(X^{(\infty,\gamma)}_{t+.})]-{\ES}[F(X^{(\infty,\gamma)}_{.})]$$
and owing to a change of variable, it is obvious that for every $t\in\gamma \mathbb{N}$,
$$\frac{1}{n}\sum_{k=1}^n F(\bx^\gm_{\gm(k-1)+t+.})-F(\bx^\gm_{\gm(k-1)+.})\xrn{n\nrn}0.$$
It follows that for every $t\in\gamma\EN$, for every $F$,
$${\ES}[F(X^{(\infty,\gamma)}_{t+.})]={\ES}[F(X^{(\infty,\gamma)}_{.})].$$
This property implies that $X^{(\infty,\gamma)}$ is stationary.\\
We now focus on the adaptation of $X^{(\infty,\gamma)}$. In this step, we need to introduce, for a subset $D$ of $\ER$ that contains $0$, the Polish space ${\cal W}_{\theta,\delta}(D)$   that denotes the completion of ${\cal C}_0^{\infty}(D,\ER^q)$ (the space of ${\cal C}^\infty$-functions $f:D\rightarrow\ER^q$ with compact support and  $f(0)=0$)  for the norm
$$\|f\|=\sup_{s,t\in D}\frac{|f(t)-f(s)|}{|t-s|^\theta(1+|t|^\delta+|s|^{\delta})}.$$
This space is convenient to obtain some Feller properties for the conditional distribution of the fractional Brownian motion given its past. More precisely, by Lemmas 4.1 to 4.3 of \cite{hairer2}, the paths of $B^H$ belong $a.s.$ to ${\cal W}_{\theta,\delta}(\ER)$ when $\theta\in(1/2,H)$ and $\theta+\delta\in(H,1)$. Furthermore, setting $B^{H,u}_t=B^H_{t+u}-B^H_u$, we also deduce from these lemmas that for every non-negative $t$ and $T$,
$${\cal P}_T(\omega,.):={\cal L}((B^{H,t+T}_s)_{s\le 0}| (B^{H,t}_s)_{s\le0}=(\omega_s)_{s\le0})$$
 is a Feller transition on ${\cal W}_{\theta,\delta}(\ER_{-})$ (which does not depend on $t$). \\
%
Let us now prove that $X^{(\infty,\gamma)}$ is adapted, $i.e.$ that for every $t\ge 0$, $(X^{(\infty,\gamma)}_s)_{s\le t}$ and
$(B^H_s)_{s\ge t}$ are independent conditionally to $(B^H_s)_{s\le t}$. One can check that it is enough to prove that for every $t\ge0$ and (arbitrary large) $T\ge0$, $(X^{(\infty,\gamma)}_s)_{s\le t}$ and
$(B^{H,t+T}_s)_{s\ge0}$ are independent conditionally to $(B^{H,t}_s)_{s\le 0}$ (using on the one hand that $(B^H_s)_{s\le t}$ is trivially $\sigma(B^H_s,s\le t)$-measurable and that for every $u\ge0$, $\sigma (B^{H,u}_s,s\le0)=\sigma(B^{H}_s,s\le u)$).  {To prove this conditional independence property, it is now enough to} show that for every $t\ge 0$,
for every  $T\ge0$, for every bounded continuous functionals $f:\bar{\cal C}^{\theta}([0,t],\ER^d)\rightarrow\ER$, $g:{\cal W}_{\theta,\delta}(\ER_{-})\rightarrow\ER$ and $h:{\cal W}_{\theta,\delta}(\ER_{-})\rightarrow\ER$
\begin{equation}\label{ekek}
\begin{split}
{\ES}[f(X^{(\infty,\gamma)}_{s},s\in[0, t])&g(B^{H,t+T}_{s},s\le 0) h(B^{H,t}_{s},s\le0)]\\
&={\ES}[f(X^{(\infty,\gamma)}_{s},s\in[0, t])\psi^{g}(B^{H,t}_{s},s\le 0)h(B^{H,t}_{s},s\le 0)]
\end{split}
\end{equation}
where $\psi^g((\omega_{s})_{s\le 0})=\ES[g(B^{H,t+T}_s,s\le 0)| (B^{H,t}_s)_{s\le 0}= (\omega_s)_{s\le 0}]={\cal P}_Tg((\omega_{s})_{s\le0})$.
Since ${\cal P}_T(\omega,.)$ is Feller, $\psi^g$ is continuous on ${\cal W}_{\theta,\delta}(\ER_{-})$.\\
Then, using the ergodicity of the increments of $B^H$, we can show as in the beginning of the proof that $(\tilde{{\cal P}}^{(n,\gamma)}(\omega))_{n\ge1}$ is  tight on $\bar{\cal C}^\theta(\ER_+,\ER^d)\times{\cal W}_{\theta,\delta}(\ER)$. Thus, there exists $a.s.$ a sequence $(n_k)$ such that
$${\ES}[f(X^{(\infty,\gamma)}_{s},s\le t)g(B^{H,t+T}_{s},s\le 0) h(B^{H,t}_{s},s\le 0)]=\lim_{k\rightarrow+\infty}\frac{1}{n_k}\sum_{k=1}^{n_k}
H_{k-1} J_k$$
and such that
 $${\ES}[f(X^{(\infty,\gamma)}_{s},s\le t)\psi^g(B^{H,t}_{s},s\le 0) h(B^{H,t}_s,s\le 0)]=\lim_{k\rightarrow+\infty}\frac{1}{n_k}\sum_{k=1}^{n_k}
H_{k-1} \ES[J_k|{\cal F}_{\gm (k-1)+t}]$$
with ${\cal F}_u=\sigma(B_s^H,s\le u)$, $H_k=f(\bx^\gm_{\gm k+s},s\le t) h(B^{H}_{\gm (k-1)+s+t}-B^H_{\gm (k-1)+t},s\le 0)$,
and $J_k={g}( B^H_{\gm (k-1)+s+t+T}-B^H_{\gm (k-1)+t+T}, s\le 0)$.
This implies that it is now enough to prove that
$$\frac{1}{n}\sum_{k=1}^n H_{k-1}\left(J_k-\ES[J_k|{\cal F}_{\gm (k-1)+t}]\right)\xrn{n\nrn}0\quad a.s.$$
This point follows from a decomposition of the above sum in martingale increments and from classical martingale arguments (see proof of Proposition 6 of \cite{cohen-panloup} for a similar argument).
\end{proof}
\subsection{Identification of limits when $\gamma \to 0^+$}
In this part we fix a $H$-fractional Brownian motion $B^H$ on $ \bar{{\mathcal C}}^\theta(\ER,\ER^q) $ 
and we consider a pair $(X^{\infty,\gamma},B^H) $ on $ \bar{{\mathcal C}}^\theta(\ER_+,\ER^d) \times \bar{{\mathcal C}}^\theta(\ER_+,\ER^q) $ 
such that for each $ \gamma > 0 $ the joint distribution is given by Proposition~\ref{prop2}.
\begin{prop}\label{gamma-0} Let $(\gamma_k) $ be a sequence converging to $ 0$ 
such that the distributions of $(X^{\infty,\gamma_k},B^H) $ are converging weakly on
$ \bar{{\mathcal C}}^\theta(\ER_+,\ER^d) \times \bar{{\mathcal C}}^\theta(\ER,\ER^q) $ to  $(X^{\infty},B^H).$ 
Then   $X^{\infty} $ is a stationary  adapted solution  to~\eqref{fractionalSDE0} in the sense of 
Definition~\ref{def:stat-sol}.
\end{prop}
\begin{proof}
Let us first introduce 
$$ \tilde{\Phi}(\alpha,\beta)_t := \alpha_0 + \int_0^t b(\tilde{\Phi}(\alpha,\beta)_{s}) ds +   \int_0^t  \sigma(\tilde{\Phi}(\alpha,\beta)_{s}) d \beta_s,$$ and remark that $ \tilde{\Phi}(\alpha,\beta) = \Phi(a,\beta),$ if $\alpha_0=a.$
We want to show that 
\begin{equation}
  \label{eq:tilde-eds}
  X^{\infty}= \tilde{\Phi}(X^{\infty},B^H)  
\end{equation}
almost surely so that $(X^{\infty},B^H) $ is a solution to~\eqref{fractionalSDE0}. 
Let us rewrite the equation with the help of two continuous operators  on $\bar{{\mathcal C}}^\theta(\ER_+,\ER^d) \times \bar{{\mathcal C}}^\theta(\ER_+,\ER^q) $ ~:
$$ \Psi(\alpha,\beta)_t= \int_0^t  b(\alpha_{s}) ds +\int_0^t  \sigma(\alpha_{s}) d \beta_s,$$
and 
$$ \Delta(\alpha)_t= \alpha_t - \alpha_0.$$
Then equation~\eqref{eq:tilde-eds} is equivalent to 
\begin{equation}
  \label{eq:tilde-eds-split}
  \Delta(X^{\infty})= \Psi(X^{\infty},B^H). 
\end{equation}
Let us also consider the discretization of $ \Psi $
$$ \Psi^{\gamma}(\alpha,\beta)_t= \int_0^t  b(\alpha_{\un{s}_{\gamma}}) d s +\int_0^t  \sigma(\alpha_{\un{s}_{\gamma}}) d \beta_s.$$
Obviously~\eqref{fractionalSDE0-disc} can be rewritten
\begin{equation}
  \label{eq:tilde-eds-split-disc}
  \Delta(X^{\infty,\gamma})= \Psi^{\gamma}(X^{\infty,\gamma},B^H). 
\end{equation}
\begin{lemme}
\label{lem:psi-conv}
Let $(\gamma_k)_{k\ge1}$ be a sequence converging to $ 0$ 
such that $(X^{\infty,\gamma_k},B^H)_{k\ge1} $ converges weakly on
$ \bar{{\mathcal C}}^\theta(\ER_+,\ER^d) \times \bar{{\mathcal C}}^\theta(\ER,\ER^q) $ to  $(X^{\infty},B^H).$
Then $\Psi^{\gamma_k}(X^{\infty,\gamma_k},B^H)$ converges weakly on
$ \bar{{\mathcal C}}^\theta(\ER_+,\ER^d) $ to  $\Psi(X^{\infty},B^H).$
\end{lemme}
\begin{proof}
Let $(\alpha,\beta) \in  \bar{{\mathcal C}}^\theta(\ER_+,\ER^d) \times \bar{{\mathcal C}}^\theta(\ER_+,\ER^q).$
A classical result concerning the discretization of Young integrals shows 
that 
$$ | \Psi(\alpha,\beta)_t- \Psi^{\gamma}(\alpha,\beta)_t| \le \|\alpha\|_{\theta,t} \|\beta\|_{\theta,t} \gamma^{2 \theta - 1} t.$$
See for instance~\cite{Coutin12}, Proposition 31 or~\cite{Young36}.
Hence for $T > 0,$
\begin{equation}
  \label{eq:int-Young-ineq}
   \| \Psi(\alpha,\beta)- \Psi^{\gamma}(\alpha,\beta)\|_{\theta,T} \le \|\alpha\|_{\theta,T} \|\beta\|_{\theta,T} \gamma^{2 \theta - 1}  T^{1-\theta}.
\end{equation}
Let $ F $ be any bounded  $K$-Lipschitz functional on $  \bar{{\mathcal C}}^\theta([0,T],\ER^d), $ 
\begin{equation}
  \label{eq:F1-ineq}
|\ES (F(\Psi(X^{\infty,\gamma_k},B^H) ) - \ES (F(\Psi(X^{\infty},B^H) )| \to 0
\end{equation}
as $ k \to \infty.$
Then 
\begin{equation}
  \label{eq:F2-ineq}
|\ES (F(\Psi^{\gamma_k}(X^{\infty,\gamma_k},B^H) ) - \ES (F(\Psi(X^{\infty,\gamma_k},B^H) )| \le K \ES ( \|X^{\infty,\gamma_k}\|_{\theta,T} \|B^H\|_{\theta,T})  T^{1-\theta} \gamma_k^{2 \theta - 1},  
\end{equation}
and using Proposition~\ref{lemme4}\emph{(ii)} the left hand side of~\eqref{eq:F2-ineq} is converging to $0$ as  $ k \to \infty.$
Combining~\eqref{eq:F1-ineq} and this last fact, we get the desired convergence in distribution. 
\end{proof}
\noindent Let us start with 
\begin{equation}
  \label{eq:concl}
\Delta(X^{\infty,\gamma_k})=\Psi^{\gamma_k}(X^{\infty,\gamma_k},B^H),
\end{equation}
and let $ k \to \infty.$ By Lemma~\ref{lem:psi-conv}, the right hand side
of~\eqref{eq:concl} converges to $ \Psi(X^{\infty},B^H) $ and the left hand side 
to $ \Delta(X^{\infty}),$ which, in turn, has the same distribution as $\Psi(X^{\infty},B^H).$\\

Now, let us prove that $X^\infty$ is stationary. It is enough to show that 
$\ES[F(X^\infty_.)]=\ES[F(X^{\infty}_{t+.})]$ for every $t\ge0$ and for every functional $F$ defined by $F(\alpha)=\prod_{k=1}^m f_i(\alpha_{t_i})$
where $f_1,\ldots, f_m$ denote Lipschitz continuous functions on $\ER^d$ and $t_1,\ldots,t_m$ belong to $\ER_+$. By Proposition \ref{prop2}, the distribution of $X^{\infty,\gamma}$ is invariant by the time-shift $(\theta_{k\gamma})$ for every $k\in\EN$
so that $\ES[F(X^\infty_.)]=\ES[F(X^{\infty}_{\un{t}+.})]$. The result follows easily by checking that for every $T>0$, 
$$\ES[\sup_{u,v\in[0,T],|u-v|\le\gamma}|X^{\infty,\gamma}_v-X^{\infty,\gamma}_u|]\xrn{\gamma\rightarrow0}0.$$
Finally, it remains to show that $(X^\infty,B^H)$ is adapted. Since $(X^{\infty,\gamma_k})$ converges in distribution
to $X^\infty$ on $\bar{\cal C}^\theta(\ER_+,\ER^d)$ and since $B^H$ belongs to ${\cal W}_{\theta,\delta}$ (with $\theta\in(1/2,H)$
and $\theta+\delta\in(H,1)$), $(X^{\infty,\gamma'_k}, B^H)$ converges to $(X^\infty,B^H)$ for $\gamma'_k$ a subsequence of $ \gamma_k.$
Then, we can let $\gamma$ go to $0$ in equality \eqref{ekek} and the result follows.
\end{proof}
\section{Simulations}\label{section6}
In this section, we give an illustration of the application of our procedure for  a one-dimensional toy equation:
$$dX_t=-X_tdt+(4+\cos(X_t))dB_t^H.$$
We propose to compute an estimation of the density of the (marginal) invariant distribution in this case.
We denote it by $\nu^H_0$. By Theorem \ref{principal1}, for every bounded continuous function $f:\ER^d\rightarrow\ER$,
$$\lim_{\gamma\rightarrow0}\lim_{n\rightarrow+\infty}{\cal P}^{(n,\gamma)}_0(\omega,f)=\nu_0^H(f).$$
\noindent The first step is to simulate the sequence $(B^H_{\gamma k}-B^H_{\gamma (k-1)})_{k=1}^n$. We use the  Wood-Chan method  (see \cite{wood}) which is based on the embedding of the covariance matrix of the fractional increments in a symmetric circulant matrix (whose eigenvalues can be computed using the Fast Fourier Transform).\\
Then, we compute $K_h\ast{\cal P}^{(n,\gamma)}_0$ where $K_h$ is the Gaussian convolution kernel defined by
$K_h(x)=\frac{1}{\sqrt{2\pi}h}\exp(-\frac{x^2}{2h})$. Note that $K_h\ast{\cal P}^{(n,\gamma)}_0(x_0)={\cal P}^{(n,\gamma)}_0(K_h(x_0-.))$, where, for a measure $\mu,$ and a $\mu$-measurable function f, we set $\mu(f)=\int fd\mu$.
In Figure \ref{figure1} is depicted the approximate density with the following choices of parameters
$$ n=10^7,\quad \gamma=0.05\quad h=0.2,\quad H=\frac{3}{4}.$$
We choose to compare it with the density of the invariant distribution when $H=1/2$. Note that in this case, the invariant distribution is (semi)-explicit (as for every ergodic one-dimensional diffusion) and is given by
$$\nu^{\frac{1}{2}}_0(dx)=\frac{M(dx)}{M(\ER)}\quad\textnormal{where}\quad M(dx)=\frac{1}{(4+\cos x)^2}\exp\left(-\int_0^x\frac{2u}{(4+\cos u)^2}du\right)dx.$$ 
We observe that the distribution when $H=3/4$ has heavier tails than in the diffusion case.
\begin{figure}[htbp]
\begin{center}
\includegraphics[width=7cm]{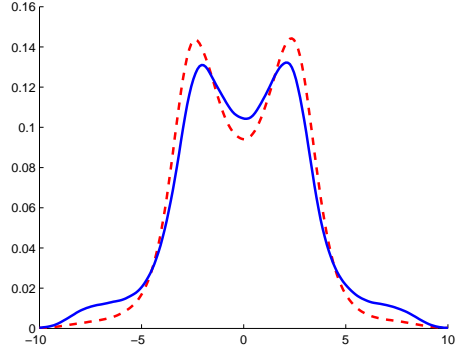}
\caption{Approximate density of $\nu_0^H$ (continuous line) compared with that of $\nu^{\frac{1}{2}}_0$ (dotted line)\label{figure1}}
\end{center}\end{figure}
Finally, in order to have a rough  idea of the rate of convergence, we depict in Figure \ref{figure2} the approximate densities for different values of $n$ keeping the other parameters unchanged.
\begin{figure}[htbp]
\begin{center}
\includegraphics[width=7cm]{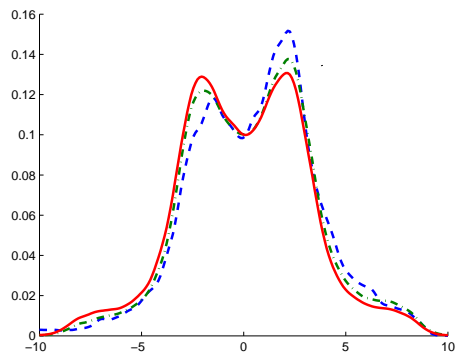}
\caption{Approximate density of $\nu_0^H$ for $n=10^5$ (dotted line), $n=10^6$ (dash-dotted line), $n=10^7$ (continuous line)\label{figure2}}
\end{center}\end{figure}
\begin{Remarque}
As mentioned before, this section is only an illustration. In fact, there are (many) numerical open questions. For  the estimation of the error, it would be necessary for a function $f$ to get some rate of convergence results for   ${\cal P}^{(n,\gamma)}_0(f)-\nu_H(f)$ (long-time error) and for $\nu^{H,\gamma}_0(f)-\nu_0^H(f)$ (discretization error) where $\nu^{H,\gamma}_0$
denotes the initial distribution of the stationary Euler scheme with step $\gamma$. Note that in the diffusion case, it can be shown  (under some appropriate assumptions that the long time error is about $(\gamma n)^{-\frac{1}{2}}$ (see \cite{bhatta82} for the corresponding result in the continuous case) whereas the discretization error is $O(\gamma)$ (see \cite{talay}, Theorem 3.3 for a similar result with the Milstein scheme). Finally, even if the Wood and Chan simulation method  is fast and exact, it requires a lot of memory because of the Fast Fourier Transform. On Matlab, for instance, this implies that we can not take $n$ greater than $2.10^7$. Thus, it could be interesting to study  some discretization schemes based on some approximations of the fBm-increments simulated, which consumes less memory.   
 \end{Remarque}
\section{Appendix}
\textbf{Proof of Proposition \ref{unicitemesinvariante}}
Let us show that   $(\bar{X}_{\gamma k})$ is  a \textit{skew-product} in the sense of~\cite{hairer09} as follows.
For a fractional Brownian $B^H$ motion on $\ER$, set for every $n\in\mathbb{Z}$ $\Delta_n^{\gamma}=B^{H}_{(n+1)\gamma}-B^H_{n\gamma}$.  Setting ${\cal W}:=(\ER^d)^{\Z_{-}}$, we then introduce the regular conditional probability $\bar{\cal P}^{\gamma}:{\cal W}\rightarrow{\cal M}_1(\ER^d)$ defined by\footnote{
Note that since $(\Delta_n^{\gamma})_{n\in\Z}$ is a stationary sequence, ${\cal L}(\Delta_1^{\gamma}|(\Delta_k^{\gamma})_{k\le0}=\omega)={\cal L}(\Delta_{n+1}^{\gamma}|(\Delta_{n+k}^{\gamma})_{k\le0}=\omega)$ for every $n\in\mathbb{Z}$.}: 
$$\bar{\cal P}^{\gamma}(\omega)={\cal L}(\Delta_1^{\gamma}|(\Delta_k^{\gamma})_{k\le0}=\omega)$$
and denote by ${\cal P}^{\gamma}$ the Feller transition on ${\cal W}$ defined  for every measurable function $f:{\cal W}\rightarrow \ER$
by ${\cal P}^{\gamma}f(\omega)=\int_{\ER^d} f(\omega\sqcup \tilde{\omega})\bar{\cal P}^{\gamma}(\omega,d\tilde{\omega})$ where for $\omega\in(\ER^d)^{\mathbb{Z}_-}$ and $\tilde{\omega}\in\ER^d$, $\omega\sqcup\tilde{\omega}=(\ldots,\omega_{_2},\omega_{_1},\omega_{0},\tilde{\omega})$.
Setting ${ \Phi}^{\gamma}(x,\tilde{\omega})=x+\gamma b(x)+\sigma(x) \tilde{\omega}$ and  $\PE_H^{\gamma}:={\cal L}((\Delta_n)_{n\le0})$,
we have defined a skew-product $({\cal W},\PE_H^{\gamma},{\cal P}^{\gamma},\ER^d,{ \Phi}^{\gamma})$ with  the transition operator ${\cal Q}^{\gamma}$ on $\ER^d\times{\cal W}$ defined by
$${\cal Q}^{\gamma} f(x,\omega)=\int f(\Phi^\gamma(x,\omega')){\cal P}^{\gamma}(\omega,d{\omega}'),$$
which describes the dynamics of the Euler scheme. \\
Then, thanks to Theorem 1.4.17 of \cite{hairer09}, uniqueness of the adapted and stationary discrete Euler scheme $(\bar{X}_{\gamma k})$ (in distribution) holds, if the skew-product $({\cal W},\PE_H^{\gamma},{\cal P}^{\gamma},\ER^d,{ \Phi}^{\gamma})$ is 
strong Feller and topologically irreducible (in the sense of Definition 1.4.6 and 1.4.7 of \cite{hairer09}).\\
First, write $\tilde{\omega}=(\tilde{\omega}^1,\ldots,\tilde{\omega}^q)$ and $\Phi^\gamma=(\Phi^\gamma_1,\ldots,\Phi^\gamma_d)$. Denote by $M^{\Phi}(x,\tilde{\omega})$ the (discrete) Malliavin covariance matrix of $\Phi$ defined by
$$\forall (x,\tilde{\omega})\in\ER^d\times\ER^d\quad\textnormal{and}(i,j)\in\{1,\ldots,d\}^2,\quad M^{\Phi}_{i,j}(x,\tilde{\omega}):=\sum_{k=1}^d \partial_{\tilde{\omega}^k}\Phi^{\gamma}_i(x,\tilde{\omega})\partial_{\tilde{\omega}^k}\Phi^{\gamma}_j(x,\tilde{\omega}).$$
Thus, $M^\Phi(x,\tilde{\omega})=(\sigma\sigma^*)(x)$ and since $\sigma^{-1}$ is bounded (and continuous), it follows that 
$x\rightarrow({\rm det}(M^\Phi)^{-1}(x,\omega)$ is bounded continuous. 
Second, the functions $D_\omega\Phi$, $D_\omega D_x \Phi$ and $D^2_\omega \Phi$ are clearly bounded continuous. Finally, the sequence $((\Delta_n^\gamma)^1)$ has a spectral density $f$ that  satisfies $\int_{-\pi}^\pi (f(x))^{-1}dx<+\infty$ (see $e.g.$ \cite{beran} for an explicit expression of $f$). Thus, it follows from Theorem 1.5.9 of \cite{hairer09} that the skew-product is strong Feller.\\
For the topological irreducibility, it is enough to show that for every $(x,\omega)\in\ER^d\times{\cal W}$, for every $(y,\varepsilon)\in\ER^d\times\ER_+^*$,  ${\cal Q}(x,\omega,B(y,\varepsilon)\times{\cal W})>0$ . Since $\sigma$ is invertible, the map $\Phi$ is controllable in the following sense: $\Phi(x,\tilde{\omega}_x)=y$ has a (unique) solution $\tilde{\omega}\in\ER^q$, for every $x,y\in\ER^d$. Furthermore, $b$ and $\sigma$
being continuous, for every $\varepsilon>0$, there exists $r_\varepsilon$ such that for every $\tilde{\omega}\in B(\tilde{\omega}_x,r_\varepsilon)$, $\Phi(x,{\tilde{\omega}})\in B(y,\varepsilon)$. Thus    
$${\cal Q}(x,\omega,B(y,\varepsilon)\times{\cal W})\ge \bar{\cal P}(\omega,B(\tilde{\omega}_x,r_\varepsilon))>0,$$
since $\bar{\cal P}(\omega,.)$ is Gaussian with positive variance. This concludes the proof.\\

{\bf Acknowledgment}

We would like to thank the anonymous referee for his/her careful reading and his/her  suggestions that helped us
to improve the first version of this article. 


\end{document}